\documentclass[reqno,twoside,11pt]{amsart}

\usepackage{amsmath,amsfonts,calrsfs,fullpage,amssymb,xcolor,verbatim,eucal,yfonts,mathrsfs}

\usepackage{mathtools}
\mathtoolsset{showonlyrefs}

 \usepackage{xcolor}
\newtheorem{Theorem}{Theorem}[section]

\newtheorem{Proposition}[Theorem]{Proposition}
\newtheorem{Lemma}[Theorem]{Lemma}
\newtheorem{Corollary}[Theorem]{Corollary}
\newtheorem{Remark}[Theorem]{Remark}
\newtheorem{claim}[Theorem]{Claim}

\newtheorem{Hypothesis}{Hypothesis}

\makeatletter
\@addtoreset{equation}{section}

\makeatother

\def\a{\alpha}

\def\la{\lambda}

\def\R{\mathbb R}

\def\z{\mathbb z}

\def\ds{\displaystyle}

\def\e{\epsilon}

\title{The inertial It\^{o} drift and its applications to particle collision}\date{}

\author[S. Cerrai]{Sandra Cerrai}
\address{Department of Mathematics\\
university of Maryland\\
}
\email{cerrai@umd.edu}
\thanks{S. Cerrai was partiallly supported by NSF Grant DMS-2348096 (2024-2027), {\em 
 Multiscale Analysis of Infinite-Dimensional Stochastic Systems}}
\author[F. Flandoli]{Franco Flandoli}
\address{Scuola Normale Superiore\\
}
\email{franco.flandoli@sns.it}
\thanks{F. Flandoli was partially supported by ERC {\em AdG NoisyFluid n.~101053472}}

\author[M. xie]{Mengzi xie}
\address{Institut f\"{u}r Mathematik\\
Technische universit\"{a}t Berlin\\
}
\email{xie@tu-berlin.de}

\subjclass[2010]{}


\begin{document}

 \begin{abstract} 
 
 The small mass $\mu$ limit of an inertial system driven by an Ornstein
Uhlenbeck fluid force, with correlation time $\epsilon$ going to zero, leads
to a first order system with an additional drift, which we call
inertial-It\^{o}-drift, depending on the limit $\alpha$ of the ratio
$\mu/\epsilon$; the drift being zero when $\alpha=0$, corresponding to the
Stratonovich integral in the limit equation, as in the Wong-Zakai theory when
applied directly to the first order system with Ornstein Uhlenbeck driver. We
discuss the application of this result to particles driven by Stokes
force;\ we identify inertial centrifugal effects and the so-called
turbophoretic effect, as examples of the inertial-It\^{o}-drift. We also
analyze concentration effects and their link with the theory of particle
collision in turbulent fluids.

 \end{abstract}

 \maketitle
 
 \tableofcontents
 
 \section{Introduction}
 
The problem of identifying the limiting equation of motion for an inertial particle subject to a fast-decorrelating random forcing is relevant in both the physical and the mathematical literature, and lies at the crossroad between the theory of small-mass (or Smoluchowski-Kramers) approximations and the Wong-Zakai theory of smooth approximations of white noise. The two limiting procedures involved, letting the mass $\mu$ of the particle tend to zero and letting the correlation time $\epsilon$ of the noise tend to zero, do not commute, and, depending on the relative magnitude of $\mu$ and $\epsilon$, the limiting equation features qualitatively different drift corrections. The aim of the present paper is to give a complete and rigorous description of this small-mass/short-correlation-time limit for a system with arbitrary state-dependent friction and noise coefficients in $\R^d$, driven by a multi-dimensional Ornstein-Uhlenbeck process, and for {\em all} possible values of the asymptotic ratio $\alpha:=\lim_{\epsilon\to 0}\mu(\epsilon)/\epsilon\in[0,+\infty]$,  and to discuss the application of these results to some physical settings.

\medskip

\noindent {\em The model.} For every $\mu,\epsilon>0$, we consider  the second-order stochastic system in $\R^d$
\begin{equation}
\label{intro-system}
\left\{\begin{array}{l}
\ds{\frac{dx}{dt}(t)=v(t),\quad x(0)=x_0,}\\[8pt]
\ds{\mu\,\frac{dv}{dt}(t)=b(x(t))-\gamma(x(t))\,v(t)+\sigma(x(t))\,z_\epsilon(t),\quad v(0)=v_0,}
\end{array}\right.
\end{equation}
where the driving process $z_\epsilon(t)$ is the rescaled Ornstein-Uhlenbeck process
\[
z_\epsilon(t)=\frac{1}{\epsilon}\int_0^t e^{-A(t-s)/\epsilon}\,B\,dw(s),
\]
solving
\[
\epsilon\,dz_\epsilon(t)=-A\,z_\epsilon(t)\,dt+B\,dw(t),\qquad z_\epsilon(0)=0,
\]
for some matrices $A\in\R^{n\times n}$, with eigenvalues having strictly positive real part, and $B\in\R^{n\times m}$, and a standard $m$-dimensional Brownian motion $w(t)$. The drift $b:\R^d\to\R^d$, the friction matrix $\gamma:\R^d\to\R^{d\times d}$, assumed uniformly positive, and the noise coefficient $\sigma:\R^d\to\R^{d\times n}$ are smooth and bounded with bounded derivatives. System \eqref{intro-system} contains two small parameters, the mass $\mu$ of the particle, which controls the inertial relaxation time, and the correlation time $\epsilon$ of the colored noise. We let both parameters go to zero simultaneously, with $\mu=\mu(\epsilon)\to 0$, as $\epsilon\to 0$, and we are interested in the limiting behavior of the slow component $x_\epsilon(t)=x_{\epsilon,\mu(\epsilon)}(t)$.

\medskip

\noindent {\em The two extreme regimes and their non-commutativity.} The non-commutativity of the two limits $\mu\to 0$ and $\epsilon\to 0$ is well known and may be illustrated already at the heuristic level. If one first sends $\epsilon\to 0$ while keeping $\mu>0$ fixed, the colored process $z_\epsilon$ tends to a white noise and, since the second-order equation produces a position  component $x$ of bounded variation (in  fact, $C^1$ in time), the resulting stochastic integral is unambiguous, and one obtains
\[
\mu\,dv(t)=(b(x(t))-\gamma(x(t))v(t))\,dt+\sigma(x(t))\,A^{-1}B\,dw(t).
\]
Subsequently, sending $\mu\to 0$ gives, after a non-trivial small-mass analysis, a first-order It\^o equation for the position with a noise-induced drift. If, instead, one first sends $\mu\to 0$ for fixed $\epsilon>0$, the velocity is slaved to the colored forcing, $v=\gamma^{-1}(b+\sigma z_\epsilon)$, so that
\[
\frac{dx(t)}{dt}=\gamma^{-1}(x(t))b(x(t))+\gamma^{-1}(x(t))\sigma(x(t))\,z_\epsilon(t),
\]
and the subsequent limit $\epsilon\to 0$ falls within the classical Wong-Zakai framework, leading to a Stratonovich equation. The two procedures yield fundamentally different stochastic differentials, and an intermediate scaling, in which $\mu$ and $\epsilon$ tend to zero at comparable rates, interpolates between them.

\medskip

\noindent {\em Earlier results.} The first rigorous treatment of the small-mass limit goes back to Freidlin~\cite{freidlin}, who considered an $\R^n$-valued system with {\em constant} friction, driven by white noise, namely the regime corresponding to $\epsilon=0$ (or, in our notation, $\alpha=+\infty$ with $z_\epsilon$ replaced by white noise from the outset). Under suitable assumptions, the position $x_\mu$ converges as $\mu\to 0$ to the unique solution of the first-order It\^o equation
\[
dx(t)=\gamma^{-1}b(x(t))\,dt+\gamma^{-1}\sigma(x(t))\,dw(t),
\]
with no additional drift. In this constant-friction setting, the friction matrix and the noise coefficient commute with the rescaled velocity, and the It\^o correction generated by the integration-by-parts trick used in the proof vanishes identically.

The setting of {\em colored} multiplicative noise was investigated by Kupferman, Pavliotis and Stuart in~\cite{pav-stu}, in the scalar case (i.e.\ $d=1$) with constant friction. Working at the level of formal asymptotic expansions and supported by extensive numerical simulations, they identified the form of the limiting equation as a function of the relative size of $\mu$ and $\epsilon$, parameterized by an exponent $\gamma\in(0,\infty)$ corresponding to $\mu\sim\epsilon^\gamma$. Three regimes emerged: for $\gamma\in(0,2)$ the limit is the bare It\^o equation, with no drift correction; for $\gamma>2$ the limit is the corresponding Stratonovich equation (i.e.\ It\^o plus the full Stratonovich correction); and for the critical value $\gamma=2$, where the relaxation time of the particle is comparable to that of the noise, a modified drift correction emerges, of neither It\^o nor Stratonovich type, depending in an essential way on the precise scaling. The arguments of~\cite{pav-stu}, however, are not fully rigorous. Convergence is established only on a formal level, and the analysis is restricted to the one-dimensional setting.

A complete rigorous justification of these formal results, in the multidimensional setting and for an infinite-dimensional driving noise, was provided shortly afterwards by Pavliotis and Stuart in~\cite{PS-2}. The model considered there is precisely a $d$-dimensional Langevin equation with constant inertial coefficient and a coloured forcing constructed as an infinite-dimensional generalised Ornstein-Uhlenbeck process. Namely, the noise has the form $v(x,t)=f(x)\eta(t)$, where $\eta$ takes values in $\ell^2$ and solves a linear SPDE
\[
d\eta(t)=-A\,\eta(t)\,dt+dw(t),
\]
driven by a $Q$-Wiener process $w(t)$ on $\ell^2$ with diagonal $A$ and $Q$, and where $f(x)\in\mathcal{L}(\ell^2,\R^d)$ is given by a series expansion in the eigenbasis of $A$. Under suitable summability conditions on the spectrum of $A$ and $Q$ - conditions which become more stringent as the exponent $\gamma$ controlling the ratio $\mu(\epsilon)/\epsilon^\gamma$ approaches and crosses the critical value $\gamma=2$ - the authors established strong $L^{2p}$ convergence in $C([0,T];\R^d)$, with explicit upper bounds on the convergence rate, in each of the three regimes $\gamma\in(0,2)$, $\gamma=2$ and $\gamma>2$. The method of proof in~\cite{PS-2} relies on  pathwise techniques and exploits in an essential way the constancy of the friction coefficient, which allows the explicit variation-of-constants representation
\[
y(t)=y_0\,e^{-t/\epsilon^\gamma}+\epsilon^{-\gamma}\int_0^te^{(s-t)/\epsilon^\gamma}\frac{v(x(s),s)}{\epsilon}\,ds+\epsilon^{-\gamma}\int_0^te^{(s-t)/\epsilon^\gamma}b(x(s))\,ds
\]
of the velocity. This explicit representation yields sharp bounds on the moments of the particle velocity and serves as the starting point for a series of integration-by-parts estimates, performed via the It\^o formula, that identify each contribution to the limiting drift. The case $\gamma=2$ is the most delicate one. There, neither It\^o-type nor Stratonovich-type estimates are sufficient, and an additional integration by parts is needed to control a resonance term whose limit produces a drift correction depending on the spectral properties of the operator $A$. The constancy of the friction matrix is used throughout the proof in~\cite{PS-2}, in particular in the explicit representation of the velocity and in the structure of the cascade of integration-by-parts identities. It is precisely this constancy that one wishes to remove in order to address the physically realistic situation of state-dependent friction.

In~\cite{hmdvw}, Hottovy, McDaniel, Volpe, and Wehr successfully treated the case of {\em variable} friction  in the small-mass limit, for white-noise forcing, corresponding in our notation to the regime $\alpha=+\infty$. They studied the multidimensional system
\[
\left\{\begin{array}{l}
\ds{dx^\mu(t)=v^\mu(t)\,dt,}\\[6pt]
\ds{\mu\,dv^\mu(t)=\left(F(x^\mu(t))-\gamma(x^\mu(t)) v^\mu(t)\right)\,dt+ \sigma(x^\mu(t))\,dw(t)}
\end{array}\right.
\]
with arbitrary smooth state-dependent matrices $\gamma$ (positive in the sense above) and $\sigma$. The main result of~\cite{hmdvw} establishes the $L^2$-convergence of $x^\mu(t)$, as $\mu\to 0$, to the solution of the It\^o equation
\[
dx(t)=\big(\gamma^{-1}(x(t))F(x(t))+S(x(t))\big)\,dt+\gamma^{-1}(x(t))\sigma(x(t))\,dw(t),
\]
where the noise-induced drift $S$ is given by
\[
S_i(x)=\partial_l\big[(\gamma^{-1})_{ij}(x)\big]\,J_{jl}(x),
\]
and the matrix $J(x)$ is the unique solution of the Lyapunov equation
\[
J(x)\gamma^\star+\gamma J(x)=\sigma\sigma^\star(x).
\]
The proof in~\cite{hmdvw} is based on a clever integration-by-parts manipulation of the second-order system, which represents the position $x^\mu$ in the form required to apply a theorem of Kurtz and Protter on the convergence of stochastic integrals along semi-martingales. As the authors point out, their method is tailored to the white-noise case. In the colored-noise version of the problem, which they discuss informally, one expects the limiting drift to take the form of a Lyapunov-type expression involving the joint dynamics of the velocity and of the Ornstein-Uhlenbeck driver. They sketch what such a result should look like for $\alpha=+\infty$ (i.e.\ when $\mu$ is much larger than $\epsilon$), but they do not provide a proof for $\alpha\in[0,+\infty)$. The present paper fills this gap.

\medskip

\noindent {\em Main result.} Our main result, Theorem~\ref{teo2}, identifies the limiting equation for $x_\epsilon(t)$ as $\epsilon\to 0$, with $\mu(\epsilon)\to 0$ and $\mu(\epsilon)/\epsilon\to\alpha\in[0,+\infty]$, for arbitrary state-dependent friction $\gamma(x)$ and noise coefficient $\sigma(x)$. The limit equation is the It\^o stochastic differential equation
\begin{equation}
\label{intro-limit}
dx^\alpha(t)=\big[(\gamma^{-1}b)(x^\alpha(t))+f_\alpha(x^\alpha(t))\big]\,dt+(\gamma^{-1}\sigma)(x^\alpha(t))\,A^{-1}B\,dw(t),
\end{equation}
where the noise-induced drift $f_\alpha$ has the explicit form
\[
[f_\alpha(x)]_i=\partial_l\big[\gamma^{-1}_{ij}(x)\big]\,[\alpha N_\alpha(x)]_{lj}+\partial_l\big[(\gamma^{-1}_{ij}\sigma_{jk})(x)\big]\,A^{-1}_{kh}\,[L_\alpha(x)]_{lh},
\]
and the matrices $L_\alpha(x)\in\R^{d\times n}$ and $N_\alpha(x)\in\R^{d\times d}$ are defined by
\[
L_\alpha(x)=\frac{1}{\alpha}\int_0^\infty e^{-\gamma(x)t/\alpha}\,\sigma(x)\,M\,e^{-A^\star t}\,dt,
\]
and
\[
N_\alpha(x)=\int_0^\infty e^{-\gamma(x)t}\,[L_\alpha(x)\sigma^\star(x)+\sigma(x)L_\alpha^\star(x)]\,e^{-\gamma^\star(x)t}\,dt,
\]
with \[M=\int_0^\infty e^{-As}BB^\star e^{-A^\star s}\,ds.\] The pair $(L_\alpha,\alpha N_\alpha)$ is precisely the off-diagonal/diagonal block structure of the covariance matrix of the unique invariant Gaussian measure of an auxiliary fast system for the variables $(u,z)$ obtained by freezing the slow variable. The coefficients $L_\alpha$ and $\alpha N_\alpha$ extend continuously to the limiting cases $\alpha=0$ and $\alpha=\infty$. Namely,  at $\alpha=0$ one finds $L_0(x)=\gamma^{-1}(x)\sigma(x)M$ and $\alpha N_\alpha\big|_{\alpha=0}=0$, recovering the Stratonovich-type drift expected from the Wong-Zakai theory applied directly to the first-order equation with colored Ornstein-Uhlenbeck driver; at $\alpha=\infty$ one finds $L_\infty(x)=0$, and the limit reduces precisely to the It\^o equation of~\cite{hmdvw}, with the drift correction $\alpha N_\alpha\big|_{\alpha=\infty}$ arising only from the kinetic-energy term. We call $f_\alpha$ the {\em inertial-It\^o drift}, as  it interpolates continuously between the Stratonovich correction at $\alpha=0$ and the drift of~\cite{hmdvw} at $\alpha=+\infty$, in agreement with what was observed numerically and formally in~\cite{pav-stu} in the constant-friction one-dimensional case.

\medskip

\noindent {\em Strategy of proof.} The proof of Theorem~\ref{teo2} proceeds along the lines of a slow-fast averaging argument, but the specific implementation depends in an essential way on the value of $\alpha$. We first establish, in Section~\ref{ss3.2}, uniform bounds on the moments of the rescaled velocity $\sqrt{\epsilon}\,u_\epsilon$ and of the rescaled noise $\sqrt{\epsilon}\,z_\epsilon$. These bounds yield uniform $L^p$-estimates on the increments of the slow process $x_\epsilon$, from which a standard tightness criterion ensures that the family of laws $\{\mathcal{L}(x_\epsilon)\}_{\epsilon\in(0,1)}$ is tight in $C([0,T];\R^d)$. By Skorokhod's representation theorem, one may then work along a subsequence as if convergence were almost sure.

The next step is the identification of the limit. Here we exploit the integration by parts that rewrite the position equation as the sum of  the expected drift $(\gamma^{-1}b)(x)\,dt$,  the It\^o term $(\gamma^{-1}\sigma)(x)\,A^{-1}B\,dw$,  two terms quadratic in the rescaled fast variables $\tilde{u}_n=\sqrt{\epsilon_n}\,u_n$ and $\tilde{z}_n=\sqrt{\epsilon_n}\,z_n$, of the form
\[
\alpha_n\int_0^t[D\gamma^{-1}(x_n(r))\tilde{u}_n(r)]\,\tilde{u}_n(r)\,dr,\qquad \int_0^t[D(\gamma^{-1}\sigma)(x_n(r))\tilde{u}_n(r)]\,A^{-1}\tilde{z}_n(r)\,dr,
\]
plus a remainder that vanishes in $L^2(\Omega;C([0,T];\R^d))$. The heart of the proof consists in showing that these two quadratic-in-fast-variables integrals converge, after time-averaging on short intervals over which the slow variable is essentially frozen, to integrals against the invariant Gaussian measure $\mu_{\alpha, x}$
 of the frozen fast system, yielding respectively to \[\int_0^t g_l^{i,j}(x(s))[\alpha N_\alpha(x(s))]_{l,j}\,ds,\ \ \ \ \ \ \int_0^t h_l^{i,j}(x(s))[L_\alpha(x(s))]_{l,j}\,ds.\]
 This is the content of Lemma~\ref{lemma5.1} and is achieved, in the regime $\alpha\in(0,+\infty)$, by a localization-and-averaging procedure. On time intervals of size $\delta_n=\mu_n\zeta_n$ (with $\zeta_n\to\infty$ chosen so that $\zeta_n^2(\mu_n\zeta_n+\mu_n+\epsilon_n)\to 0$), one freezes the slow variable $x_n$ to its initial value on the interval, replaces the fast process $\tilde{u}_n$ by the corresponding linear OU-type process $\bar{u}_n$ driven by frozen coefficients, and exploits the exponential ergodicity of the resulting frozen system to substitute time averages with averages against the unique Gaussian invariant measure $\mu_{\alpha,x}$. The covariance of this invariant measure is computed explicitly via the associated Lyapunov equation, yielding the matrices $L_\alpha(x)$ and $N_\alpha(x)$ that appear in the limiting drift.

The cases $\alpha=0$ and $\alpha=\infty$ require different arguments. When $\alpha_n\to 0$ the time intervals $[k\delta_n,(k+1)\delta_n]$ on which one would localize the slow process become too short to allow the averaging argument to take effect. Instead, one exploits the smoothness of the colored process $z_\epsilon$ relative to the particle relaxation scale and identifies the limit drift directly as $L_0(x)=\gamma^{-1}\sigma(x)M$, recovering the Stratonovich integral by an integration-by-parts argument on the colored noise. When $\alpha_n\to+\infty$ one shows by a comparison argument that the slow position $x_n$ stays close, in $L^2(\Omega;C([0,T];\R^d))$, to the position $\bar{x}_n$ associated with an auxiliary system in which the colored forcing $\sigma(x_n)z_n$ is replaced by the white-noise forcing $\sigma(\bar{x}_n)\,A^{-1}B\,dw(t)$. Since this auxiliary system is precisely the one studied in the small-mass limit problem treated in~\cite{hmdvw}, the result follows from theirs.

\medskip

\noindent {\em Applications.} Section~\ref{subsect examples drift} is devoted to applications of the main result to the motion of inertial particles in turbulent fluids, modelled via a Stokes-type drag of the form
\[
-c_0\,k_T(x(t))\Big(v(t)-\bar{u}(x(t))-\sum_{k\in K}\xi_k(x(t))z^k(t)\Big),
\]
where $\bar{u}(x)$ is the slow large-scale mean velocity, $k_T(x)$ is the local turbulent kinetic energy (which, by Smagorinsky-type closures, plays the role of an effective space-dependent viscosity coefficient), and $\sum_k\xi_k(x)z^k$ is an Ornstein-Uhlenbeck synthetic model of the intermediate-scale fluctuations. After applying Theorem~\ref{teo2} we identify two physically meaningful contributions to the inertial-It\^o drift. The first is a { centrifugal drift} $\sum_kD\xi_k(x)\xi_k(x)$, which captures the effect, surviving in the zero-mass limit, of the inertia of the particle subject to the rapid alternating rotations induced by vortical $\xi_k$. Even though the concept of centrifugal force is purely Newtonian, its memory persists in the drift and pushes particles outward from the centers of rotation. We illustrate this with a single-vortex example in $\R^2$, where the drift indeed produces a radial outward force, and with a multi-vortex example where the centrifugal effect leads to {\em concentration} of particles in the regions between vortices, a phenomenon directly connected to the theory of preferential concentration. The second is the {\em turbophoretic drift}, which arises whenever the local turbulent kinetic energy $k_T(x)$ is non-uniform. Inertial particles drift towards regions of lower turbulent intensity, providing a stochastic-analytic derivation of the classical turbophoresis effect~\cite{Reeks,Johnson}. Finally, we discuss the implications of these effects for the theory of particle aggregation and collision in turbulent flows, where centrifugal and turbophoretic drifts play a fundamental role in determining clustering patterns and collision rates.
 
\section{Assumptions and main results}
For every $\mu, \epsilon>0$, we  consider the following  system
 \begin{equation}
 \label{system1}
 \left\{\begin{array}{l}
\ds{\frac{dx}{dt}(t)=v(t),\ \ \ \ x(0)=x_0,}\\[10pt]
\ds{\frac{dv}{dt}(t)=\frac{1}{\mu}\left(b(x(t))-\gamma(x(t))v(t)+\sigma(x(t))\,z_\epsilon(t)\right),\ \ \ \ v(0)=v_0,}\end{array}\right.
	\end{equation}
	where
	\[z_\e(t)=\frac 1\epsilon\int_0^{t} e^{-\frac 1\epsilon A(t-s)}B\,dw(s),\ \ \ \ \ t\geq 0,\]
is the solution of the problem
\[\e\,dz(t)=- Az(t)\,dt+B\, d w(t),\ \ \ \ z(0)=0,\]
for some $m$-dimensional standard Brownian motion $w(t)=(w_1(t),\ldots,w_m(t))$, defined on a filtered probability space $(\Omega, \mathcal{F}, \{\mathcal{F}_t\}_{t\geq 0}, \mathbb{P})$,  with  $A \in\,\mathbb{R}^{n\times n}$ and  $B \in\,\mathbb{R}^{n\times m}$.
It is well known  that for every $p\geq 1$ and $T>0$
\begin{equation}\label{smf26}
\sup_{t \in\,[0,T]}\mathbb{E}\,\vert \sqrt{\e}z_\epsilon(t)\vert^{2p}\leq c_{T,p}	
\end{equation}

In what follows we will assume  the following conditions.
\begin{Hypothesis}
\label{hyp}
\begin{enumerate}
\item[1.]  The mapping  $\sigma:\mathbb{R}^d\to \mathbb{R}^{d\times n}$ is bounded and twice continuously differentiable, 	 with bounded derivatives. 
\item[2.] The mapping  $b:\mathbb{R}^d\to\mathbb{R}^d$ is continuously differentiable, with bounded derivative.
\item[3.] The mapping $\gamma:\mathbb{R}^d\to\mathbb{R}^{d\times d}$ is continuously differentiable with bounded derivative, and there exist $0<\gamma_0<\gamma_1$ such that
\begin{equation}\label{smf12}\gamma_0 \,I_{\mathbb{R}^d}\leq \frac {\gamma(x)+\gamma^\star(x)}2\leq \gamma_1\, I_{\mathbb{R}^d},\ \ \ \ x \in\,\mathbb{R}^d.\end{equation}
\item[4.] The eigenvalues of the matrix $A \in\,\mathbb{R}^{n\times n}$ have all strictly positive real part.
\end{enumerate} 

\end{Hypothesis}
As shown  below, the Lipschitz-continuity of the coefficients $\sigma$, $b$ and $\gamma$ imply that system \eqref{system1} is well posed in $L^2(\Omega;C^1([0,T];\mathbb{R}^d)\times C([0,T];\mathbb{R}^d))$, for every $T>0$.
\begin{Theorem}\label{teo1}
Under Hypothesis \ref{hyp},  for every $\epsilon, \mu >0$ and $x_0, v_0 \in\,\mathbb{R}^d$  system \eqref{system1} has a unique solution $(x_{\e,\,\mu}, v_{\e,\,\mu}) \in\,L^2(\Omega;C^1([0,T];\mathbb{R}^d)\times C([0,T];\mathbb{R}^d))$, for every $T>0$.	 
\end{Theorem}

\begin{proof} For every $\e, \mu>0$, system \eqref{system1} can be rewritten as 
\begin{equation}\label{smf23}\frac{dy}{dt}(t)=M_{\epsilon,\mu}( y(t)),\ \ \ \ \ y(0)=\begin{pmatrix}x_0\\v_0\end{pmatrix},\end{equation}
where
\[y(t):=(
x(t), v(t)),\ \ \ \ \ M_{\epsilon,\mu}(x,v):=\frac 1\mu\,(\mu\,v ,b(x)-\gamma(x)v+\sigma(x)z_\epsilon).
\]
The mapping $M_{\epsilon,\mu}$ is locally Lipschitz-continuous in $\mathbb{R}^d\times \mathbb{R}^d$, hence there exists a local solution $y_{\e,\mu}=(x_{\epsilon, \mu}, v_{\epsilon, \mu})$ for equation \eqref{smf23} in $ C([0,\tau_{\epsilon,\mu}];\mathbb{R}^d\times \mathbb{R}^d)$, for some stopping time $\tau_{\epsilon,\mu}$.
Now
\[\begin{array}{l}
\ds{\frac 12\frac{d}{dt}\vert v_{\epsilon,\mu}(t)\vert^2=\frac 1\mu b(x_{\epsilon,\mu}(t))\cdot v_{\epsilon,\mu}(t)+\frac 1\mu \gamma(x_{\epsilon,\mu}(t))v_{\epsilon,\mu}(t)\cdot v_{\epsilon,\mu}(t)+\frac 1\mu\sigma( x_{\epsilon,\mu}(t))z_\e(t)\cdot v_{\epsilon,\mu}(t)}\\[10pt]
\ds{\quad \quad\quad \quad\quad \quad\quad \quad \leq -\frac{\gamma_0}{2\mu}\vert v_{\epsilon,\mu}(t)\vert^2+\frac{c}{\mu}\left(1+\vert	 z_\epsilon(t)|^2+\vert x_{\epsilon,\mu}(t)\vert^2\right).}
\end{array}\]
Therefore, for every $t\leq \tau_{\epsilon,\mu}$
\begin{equation}\label{smf24}\vert v_{\epsilon,\mu}(t)\vert^2\leq e^{-\frac{\gamma_0}{\mu}t}\vert v_0\vert^2+\frac{c}{\mu}\int_0^t e^{-\frac{\gamma_0}{\mu}(t-s)}\left(1+\vert x_{\epsilon,\mu}(s)\vert^2+\vert	 z_\epsilon(s)|^2\right)\,ds.\end{equation}
Now, since
\[\frac{d}{dt}\vert x_{\epsilon,\mu}(t)\vert^2\leq \vert x_{\epsilon,\mu}(t)\vert^2+\vert v_{\epsilon,\mu}(t)\vert^2,\]
we get
\[\vert x_{\epsilon,\mu}(t)\vert^2\leq c_T\left(\vert x_0\vert^2+\int_0^t \vert v_{\epsilon,\mu}(s)\vert^2\,ds\right),\ \ \ \ \ t\leq \tau_{\e,\mu}.\]
Due to \eqref{smf24}, this implies
\[\vert x_{\epsilon,\mu}(t)\vert^2\leq c_T\left(1+\vert x_0\vert^2+\vert v_0\vert^2+\sup_{t \in\,[0,T]}\vert	 z_\epsilon(t)|^2\right)+c_T\int_0^t 
\vert x_{\epsilon,\mu}(s)\vert^2\,ds,\]
and hence
\begin{equation}\label{smf160}\begin{array}{l}
\ds{	\sup_{t \in\,[0,\tau_{\epsilon,\mu}]}\vert x_{\epsilon,\mu}(t)\vert^2\leq c_T\Big(1+\vert x_0\vert^2+\vert v_0\vert^2+\sup_{t \in\,[0,T]}\vert	 z_\epsilon(t)|^2\Big).}
\end{array}\end{equation}
This implies that $\tau_{\epsilon,\mu}=T$, $\mathbb{P}$-a.s., so that for every $T>0$ the unique solution $y_{\epsilon,\mu}$ belongs to $C([0,T];\mathbb{R}^d\times \mathbb{R}^d)$, $\mathbb{P}$-a.s. Moreover, \eqref{smf160} also implies 
\[\mathbb{E}\vert x_{\epsilon, \mu}\vert^p _{C([0,T];\mathbb{R}^d)}<\infty,\]
for every $p\geq 1$, and, due to \eqref{smf24} this gives 
\[\mathbb{E}\vert v_{\epsilon, \mu}\vert^p _{C([0,T];\mathbb{R}^d)}<\infty.\]

	\end{proof}

Once proved that for every fixed $\epsilon, \mu>0$  system \eqref{system1} has a unique solution $(x_{\epsilon, \mu}, v_{\epsilon, \mu})$,  we want to   study the limiting behavior of $x_{\epsilon,\,\mu}$, as both $\e$ and $\mu$ go to zero.  
In what follows we will assume $\mu=\mu(\e)$
with
\[\lim_{\e \to 0}\mu(\epsilon)=0,\]
and, for every $\epsilon>0$ and $t \in\,[0,T]$,  we will define \[x_\e(t)=x_{\epsilon, \mu(\epsilon)}(t),\ \ \ \ \ v_{\epsilon}(t):=v_{\epsilon, \mu(\epsilon)}(t).\]

\begin{Theorem} \label{teo2}
Assume Hypothesis \ref{hyp}, and for every $\a \in\,[0,+\infty]$, denote by  $x^\alpha$  the unique solution of the problem
\begin{equation}
\label{main-1}
\begin{array}{l}
\ds{dx^\alpha(t)=(\gamma^{-1}b)(x^\alpha(t))+f_\alpha(x^\alpha(t))\,dt+(\gamma^{-1}\sigma)(x^\alpha(t))\, B\,dw(t),}	\end{array}
	\end{equation}
where $f_\alpha:\mathbb{R}^d\to\mathbb{R}^d$ is the vector field \begin{equation}\label{smf91}
\begin{array}{l}
\ds{	[f_\alpha(x)]_{i}=\partial_l\,\gamma^{-1}_{i,j}(x)	[\alpha N_\alpha(x)]_{l,j}+\partial_l\,(\gamma^{-1}_{i,j}\sigma_{j,k})(x)A^{-1}_{k,h}[L_\alpha(x)]_{l,h},}
\end{array}
\end{equation}
and  $N_\alpha \in\,\mathbb{R}^{d\times d}$ and $L_\alpha\in\,\mathbb{R}^{d\times n}$ are the matrices 
\begin{equation}
\label{smf9-tris}
L_{\alpha}(x):=\frac 1\alpha \int_0^\infty e^{-\frac 1\alpha \gamma(x)t}\sigma(x)Me^{-A^\star t}\,dt,	
\end{equation}
and
\begin{equation}
\label{smf9-quart}
N_\alpha(x):=\int _0^\infty e^{-\gamma(x) t}\,[L_{\alpha}(x)\, \sigma^\star(x)+\sigma(x) L^\star_{\alpha}(x)]\,e^{-\gamma^\star(x)t}\,dt,	
\end{equation}
with
\begin{equation}
\label{smf10-bis}
M:=\int_0^{+\infty} e^{-A s}BB^\star e^{-As}\,ds.	
\end{equation}
Then, if
\[\lim_{\epsilon\to 0}\frac{\mu(\epsilon)}{\epsilon}=\alpha,\]
for every $\eta>0$ and $T>0$ we have
\begin{equation}
\label{main}
\lim_{\e\to 0}\mathbb{P}\left(	\sup_{t \in\,[0,T]}|x_{\epsilon}(t)-x^\alpha(t)|>\eta\right)=0.
\end{equation}
\end{Theorem}

\begin{Remark}
{\em 	 We can easily check that for every $\alpha \in\,(0,+\infty)$ and $x \in\,\mathbb{R}^d$
\[L_\alpha(x)=\int_0^\infty e^{-\gamma(x) t}\sigma(x)Me^{-\alpha A^\star  t}\,dt.\]
Thus, we can continuously extend $L_\alpha(x)$ to the case $\alpha=0$ and $\alpha=\infty$ by setting
\begin{equation}\label{smf98}L_0(x)=\gamma^{-1}(x)\sigma(x)M,\ \ \ \ \ L_\infty(x)=0.\end{equation}
     Moreover, we have
     \[\begin{array}{l}
\ds{\alpha N_\alpha(x)=\int_0^\infty e^{-\gamma(x)t}\int_0^\infty e^{-\frac 1\alpha \gamma(x) s}\sigma(x)M e^{-A^\star s}\,ds\,\sigma^\star(x)e^{-\gamma^\star(x) t}\,dt}\\[14pt]
\ds{\quad \quad \quad \quad \quad \quad +\int_0^\infty e^{-\gamma(x)t}\sigma(x)\int_0^\infty e^{-A s}M^\star \sigma^\star (x)e^{-\frac 1\alpha \gamma^\star(x) s}\,ds\,e^{-\gamma^\star(x) t}\,dt.
}	
\end{array}
\]
Hence, we can continuously extend $\alpha N_\alpha(x)$ for $\alpha=0$ and $\alpha=\infty$ as
\begin{equation}\label{smf99}
0 N_0(x)=0
\end{equation}
and 
\begin{equation}\label{smf100}\infty N_{\infty}(x)=\int_0^\infty e^{-\gamma(x)t}\sigma(x)\left[M(A^\star)^{-1}+A^{-1}M^\star \right]\sigma^\star(x)e^{-\gamma^\star(x) t}\,dt.	
\end{equation}

}
\hfill{$\Box$}	
\end{Remark}

\subsection{The  case when $\gamma$ is scalar}

We assume here
\[\gamma(x)=\lambda(x)\,\text{Id}_{\mathbb{R}^d},\ \ \ \ x \in\,\mathbb{R}^d,\]
for some function $\lambda:\mathbb{R}^d\to\mathbb{R}$ which is twice continuosly differentiable, with bounded derivatives, such that
\[0<\lambda_0\leq \lambda(x)\leq \lambda_1,\ \ \ \ x \in\,\mathbb{R}^d.\]
Moreover, we take
\[A=\text{Id}_{\mathbb{R}^n}.\] 
It is possible to check that in this case, for every $\alpha \in\,[0,+\infty]$ and $x \in\,\mathbb{R}^d$,  the matrices $L_\alpha$ and $N_\alpha$ are given by
\begin{equation}
\label{smf150-*}
L_\alpha(x)=\frac{1}2\,(\lambda(x)+\alpha)^{-1}\,\sigma(x)\,BB^\star,	
\end{equation}
and
\begin{equation}
\label{smf151-*}
N_\alpha(x)=\frac{1}2(\lambda(x)(\lambda(x)+\alpha))^{-1}\,(\sigma(x)B)(\sigma(x)B)^\star.
\end{equation}
This implies that
\begin{equation}\label{smf153}
\begin{array}{l}\ds{f_\alpha(x)=\frac 12 \,\frac{\lambda(x)}{\lambda(x)+\alpha}\,\text{Tr}\,\left[D(\lambda^{-1}\sigma)(x)\,(\lambda^{-1}\sigma)(x)BB^\star\right]-\frac 12 \,\frac{\alpha}{\lambda(x)+\alpha}\,(\sigma(x)B)(\sigma(x)B)^\star\frac{\nabla \lambda(x)}{\lambda^3(x)}}\\[14pt]
\ds{\quad =\frac 12 \,\text{Tr}\left[D((\lambda^{-1}\sigma)(x)B)\,((\lambda^{-1}\sigma)(x)B)\right]-\frac 12 \,\frac{\alpha}{\lambda(x)+\alpha}\,\text{Tr}\left[D((\lambda^{-1}\sigma)(x)B)\,((\lambda^{-1}\sigma )(x)B)\right]}\\[14pt]
\ds{\quad \quad \quad \quad \quad \quad \quad \quad \quad \quad \quad \quad \quad \quad \quad -\frac 12 \,\frac{\alpha}{\lambda(x)+\alpha}\,(\sigma(x)B)(\sigma(x)B)^\star\frac{\nabla \lambda(x)}{\lambda^3(x)}.}	
\end{array}
\end{equation}
In particular, the limiting equation \eqref{main-1} becomes
\begin{equation}
\label{smf155}
\begin{array}{l}
\ds{dx^\alpha(t)=(\lambda^{-1}b)(x^\alpha(t))-\frac 12 \,\frac{\alpha}{\lambda(x^\alpha(t))+\alpha}\,h(x^\alpha(t))\,dt+(\lambda^{-1}\sigma)(x^\alpha(t))\circ B\,dw(t),}	\end{array}
	\end{equation}	
where
\[h(x)=\text{Tr}\,\left[D((\lambda^{-1}\sigma)(x)B)\,((\lambda^{-1}\sigma)(x)B)\right]+(\sigma(x)B)(\sigma(x)B)^\star\frac{\nabla \lambda(x)}{\lambda^3(x)}.\]

Thus, when $\alpha=0$, the limiting equation is simply the first order equation with drift $\lambda^{-1}b$ and diffusion $\lambda^{-1}\sigma$, where the stochastic differential has to be understood in Stratonovich sense. On the other hand, when $\a=\infty$, the limiting equation is a first order equation with drift $\lambda^{-1}b$ and diffusion $\lambda^{-1}\sigma$, with the stochastic differential of It\^o type, where the extra drift
\[-\frac 12\,(\sigma(x)B)(\sigma(x)B)^\star\frac{\nabla \lambda(x)}{\lambda^3(x)}\]
emerges. In all intermediate cases $\alpha \in\,(0,+\infty)$  a sort of {\em interpolation} happens between the two extreme cases, through the term
\[-\frac 12 \,\frac{\alpha}{\lambda(x)+\alpha}\,h(x).\]

\medskip

Next, we consider the special case when
\[\sigma(x)=\lambda(x)\,\xi(x),\ \ \ \ \ \ x \in\,\mathbb{R}^d,\] 
for some mapping $\xi:\mathbb{R}^d\to \mathbb{R}^{d\times n}$ which is bounded and twice continuously differentiable, 	 with bounded derivatives.
In this case system \eqref{system1} can be rewritten as
\begin{equation}
 \label{system1-special}
 \left\{\begin{array}{l}
\ds{\frac{dx}{dt}(t)=v(t),\ \ \ \ x(0)=x_0,}\\[10pt]
\ds{\mu(\epsilon)\,\frac{dv}{dt}(t)=-\lambda(x(t))\,\left(v(x(t))-\xi(x(t))z_\epsilon(t)\right)+b(x(t)),\ \ \ \ v(0)=v_0,}\end{array}\right.
	\end{equation}
	where
	\[z_\e(t)=\frac 1\epsilon\int_0^{t} e^{-\frac {(t-s)}{\epsilon}}B\,dw(s),\ \ \ \ \ t\geq 0.\]
Due to \eqref{smf150-*} and \eqref{smf151-*}, the matrices $L_\alpha$ and $N_\alpha$,  are now given by
\begin{equation}
\label{smf150}
L_\alpha(x)=\frac 12\,\frac{\lambda(x)}{\lambda(x)+\alpha}\,\xi(x)\,BB^\star,	
\end{equation}
and
\begin{equation}
\label{smf151}
N_\alpha(x)=\frac 12\,\frac{\lambda(x)}{\lambda(x)+\alpha}\,\,(\xi(x)B)(\xi(x)B)^\star.
\end{equation}
Moreover,
\begin{equation}\label{smf153-bis}
\begin{array}{l}\ds{f_\alpha(x) =\frac 12 \,\text{Tr}\,\left[D(\xi(x)B)\,(\xi(x)B)\right]}\\[14pt]
\ds{\quad \quad \quad \quad \quad -\frac 12 \,\frac{\alpha}{\lambda(x)+\alpha}\left[\text{Tr}\,\left[D(\xi(x)B)(\xi(x)B)\right]+(\xi(x)B)(\xi(x)B)^\star\frac{\nabla \lambda(x)}{\lambda(x)}\right],}	
\end{array}
\end{equation}
and the limiting equation \eqref{main-1} becomes
\begin{equation}
\label{smf155-bis}
\begin{array}{l}
\ds{dx^\alpha(t)=(\lambda^{-1}b)(x^\alpha(t))-\frac 12 \,\frac{\alpha}{\lambda(x^\alpha(t))+\alpha}\,h(x^\alpha(t))\,dt+\xi(x^\alpha(t))\circ B\,dw(t),}	\end{array}
	\end{equation}	
with
\[h(x)=\text{Tr}\,\left[D(\xi(x)B)(\xi(x)B)\right]+(\xi(x)B)(\xi(x)B)^\star\frac{\nabla \lambda(x)}{\lambda(x)}.\]

\section{Preliminaries for the proof of Theorem \ref{teo2}}

If  we define 
\[u(t):= v(t)-(\gamma^{-1}b)(x(t)),\ \ \ \ \ \ \ \ t\geq 0,\]
system \eqref{system1} can be rewritten as
\begin{equation}
 \label{system1-bis}
 \left\{\begin{array}{l}
\ds{\frac{dx}{dt}(t)=u(t)+(\gamma^{-1}b)(x(t)),\ \ \ \ x(0)=x_0,}\\[10pt]
\ds{\mu(\epsilon)\,\frac{du}{dt}(t)=-\gamma(x(t))\,u(t)+\sigma(x(t))\,z_\epsilon(t)-\mu(\e)\,R(x(t),u(t)),\ \ \ \ u(0)=v_0-(\gamma^{-1}b)(x_0),}\end{array}\right.
	\end{equation}
	where
\begin{align}\label{smf16}
R(x,u)=&\left([D\gamma^{-1}(x)(\gamma^{-1}b)(x)]b(x)+\gamma^{-1}(x)Db(x)(\gamma^{-1}b)(x)\right)\\[10pt]
& +\left(\gamma^{-1}(x)Db(x)u+[D\gamma^{-1}(x)u]b(x)\right)=:S_1(x)+S_2(x)u.
\end{align}
Notice that the mapping $R:\mathbb{R}^d\times \mathbb{R}^d\to\mathbb{R}^d$ is continuous and for every $x, u \in\,\mathbb{R}^d$
\begin{equation}
\label{R(x,u)} 
\vert S_1(x)\vert\leq c\, \left(1+|x|^2\right),\ \ \ \ \ \vert S_2(x)\vert_{\mathbb{R}^{d\times d}}\leq c\,(1+|x|).\end{equation}
Moreover, the mapping $S_2:\mathbb{R}^d\to\mathbb{R}$ is Lipschitz continuous and the mapping $S_1:\mathbb{R}^d\to\mathbb{R}$ is locally Lipschitz continuous, with
\begin{equation}
\label{smf108}
\vert S_1(x)-S_1(y)\vert	_{\mathbb{R}^d}\lesssim (1+\vert x\vert+\vert y\vert )\vert x-y\vert.
\end{equation}
\subsection{The fast system}
For every $\epsilon, \mu>0$ we define
\[\hat{z}_\epsilon(t):=\sqrt{\epsilon}\, z_\e(\e t),\ \ \ \ \   \ \hat{u}_{\epsilon}(\e t):=\sqrt{\epsilon}\,\,u_{\epsilon}(\e t),\ \ \ \ \ \ \hat{x}_{\e}(t)=x_{\e}(\e t).\]
It is immediate to check that
\begin{equation}
 \label{smf2}
 \left\{\begin{array}{l}
 \ds{\frac{d\hat{x}_{\epsilon}}{dt}(t)=\sqrt{\e}\,\hat{u}_{\epsilon}(t)+\e\,(\gamma^{-1}b)(\hat{x}_{\epsilon}(t)),}\\[10pt]
\ds{\frac{d\hat{u}_{\epsilon}}{dt}(t)=\frac{\epsilon}{\mu(\e)}\left[-\gamma(\hat{x}_{\epsilon}(t))\,\hat{u}_{\epsilon}(t)+\sigma(\hat{x}_{\epsilon}(t))\,\hat{z}_\epsilon(t)\right]-\e\left[\,\sqrt{\e}\,S_1(\hat{x}_{\epsilon}(t))+ S_2(\hat{x}_{\epsilon, \mu}(t))\,\hat{u}_{\epsilon}(t)\right],}\\[10pt]
\ds{d\hat{z}_\epsilon(t)=-A\hat{z}_\epsilon(t)\,dt+Bdw(t).}\end{array}\right.
	\end{equation}


Next, for every $\alpha \in\,(0,+\infty)$ and $x, u \in\,\mathbb{R}^d$, we introduce the system
\begin{equation}
 \label{smf4}
 \left\{\begin{array}{l}
\ds{du(t)=\frac{1}{\alpha}\left[-\gamma(x)\,u(t)+\sigma(x)\,z(t)\right]\,dt,\ \ \ \ \ u(0)=u,}\\[10pt]
\ds{dz(t)=-Az(t)\,dt+Bdw(t),\ \ \ \ \ z(0)=z.}\end{array}\right.
	\end{equation}
	In what follows, we will denote its solution by
	\[(u_\alpha^{x,u}(t),z^z(t)).\]
\begin{Lemma}
For every $\alpha \in\,(0,\infty)$ and $x \in\,\mathbb{R}^d$, system \eqref{smf4} has a unique invariant measure  given by the centered Gaussian measure $\mu_{\alpha,x}$ having covariance matrix
\begin{equation}\label{smf11}Q_{\alpha}(x)=\begin{pmatrix}
	\ds{N_\alpha(x) } & L_{\alpha}(x)\\
	&\\
	L^\star_{\alpha}(x)  & \ds{M}
\end{pmatrix},\end{equation}
where $M$ is the matrix
\begin{equation}
\label{smf10}
M:=\int_0^{+\infty} e^{-A s}BB^\star e^{-As}\,ds,	
\end{equation}
and for every $\alpha \in\,(0,+\infty)$ and $x \in\,\mathbb{R}^d$, $L_\alpha(x)$ and $N_\alpha(x)$ are the matrices defined by 
\begin{equation}
\label{smf9}
L_{\alpha}(x):=\frac 1\alpha \int_0^\infty e^{-\frac 1\alpha \gamma(x)t}\sigma(x)Me^{-A^\star t}\,dt,	
\end{equation}
and
\begin{equation}
\label{smf9-bis}
N_\alpha(x):=\int _0^\infty e^{-\gamma(x) t}\,[L_{\alpha}(x)\, \sigma^\star(x)+\sigma(x) L^\star_{\alpha}(x)]\,e^{-\gamma^\star(x)t}\,dt.	
\end{equation}

\end{Lemma}

\begin{proof}
System 	\eqref{smf4} can be written as
\[d\Lambda(t)=\Gamma_{\alpha}(x)\,\Lambda(t)+\Sigma\,dw(t),\]
where
\[\Lambda(t):=\begin{pmatrix}
u(t)\\ \\z(t)	
\end{pmatrix}, \ \ \ \ \ \ \Gamma_{\alpha}(x):=\begin{pmatrix}
-\frac 1{\alpha}\, \gamma(x)  &   \frac 1{\alpha}\,\sigma(x)\\
&\\
0  &  -A	
\end{pmatrix},\ \ \ \ \ \Sigma:=\begin{pmatrix}
0\\  \\ B	
\end{pmatrix}.\]

Since we are assuming that the eigenvalues of both $\gamma(x)$ and $A$ have strictly positive real part, we have that the eigenvalues of the matrix $\Gamma_{\alpha}(x)$ have all negative real part. This implies that the matrix 
\[Q_{\alpha}(x)=\int_0^{+\infty} e^{t\, \Gamma_{\alpha}(x)}\, \Sigma \Sigma^\star\, e^{t\, \Gamma^\star_{\alpha}(x)}\,dt\]
is well defined and the Gaussian measure $\mu_{\alpha, x}:=\mathcal{N}(0,Q_{\alpha}(x))$ is the 
unique invariant measure for system  \eqref{smf4}.
Moreover, $Q_{\alpha}(x)$ can be determined as the unique solution $J$ of the Lyapunov equation
\begin{equation}\label{smf5}\Gamma_{\alpha}(x)\,J+J\, \Gamma_{\alpha}(x)^\star=-\Sigma \Sigma^\star.\end{equation}
Thus, in what follows we  determine the  blocks of the solution matrix 
\[J=\begin{pmatrix}
J_1  &   J_2\\
&\\
J_3 &  J_4	
\end{pmatrix}, \ \ \ \ \ \text{with}\ J_3=J_2^\star,\] by solving \eqref{smf5}. We start with $J_4$, which has to solve
\[(-A) J_4+J_4(-A)^\star=-BB^\star.\]
Hence
\begin{equation}\label{smf6}J_4=\int_0^{+\infty} e^{-A t}BB^\star e^{-A^\star t} \,dt.\end{equation}
For the block $J_2$, we can easily check that it has to solve the equation
\[-\frac 1\alpha \gamma(x)\, J_2+J_2 (-A)^\star=-\frac 1\alpha \sigma(x)\, J_4.\]
Thus, we get
\begin{equation}\label{smf7}J_2(x)=\frac 1\alpha \int_0^\infty e^{-\frac 1\alpha \gamma(x)t}\sigma(x)\left[\int_0^{+\infty} e^{-A s}BB^\star e^{-A^\star s}\,ds\right]e^{-A^\star t}\,dt.\end{equation}
Finally, for the block $J_1$, it solves the equation
\[\gamma(x)\,J_1+J_1\gamma^\star(x)=\sigma(x) J_2^\star+J_2 \sigma^\star(x),\]
so that
\begin{equation}\label{smf8}J_1(x)=\int _0^\infty e^{-\gamma(x) t}[\sigma(x) J_2^\star+J_2\, \sigma^\star(x)]e^{-\gamma^\star(x)t}\,dt.\end{equation}
By putting together \eqref{smf6}, \eqref{smf7} and \eqref{smf8}, we obtain \eqref{smf11}, with entries given by \eqref{smf10}, \eqref{smf9} and \eqref{smf9-bis}.
\end{proof}

\begin{Remark}
{\em      In case $A=I_{\mathbb{R}^n}$, we have  
\[M=-\frac 12 BB^\star,\]
so that
\[\begin{array}{l}
\ds{L_{\alpha}(x)=-\frac 1{2\alpha} \int_0^\infty e^{-t}e^{-\frac 1\alpha \gamma(x)t}\,dt\,\sigma(x)BB^\star=
-\frac 12 (\alpha\,I_{\mathbb{R}^d}+\gamma(x))^{-1}\sigma(x)BB^\star,}
\end{array}
\]
and
\[\begin{array}{l}
\ds{N_\alpha(x)=-\frac{1}{2}\int_0^\infty e^{-\gamma(x) t} \hat{N}_{\alpha}(x) e^{-\gamma^\star(x)t}\,dt,}	
\end{array}
\]
where
\[\hat{N}_\alpha(x):=(\alpha\,I_{\mathbb{R}^d}+\gamma(x))^{-1}\sigma(x)BB^\star\sigma^\star(x)+\sigma(x)BB^\star\sigma^\star(x)(\alpha\,I_{\mathbb{R}^d}+\gamma^\star(x))^{-1}.\]

	} 
 \hfill{$\Box$}
 \end{Remark}

In what follows, we will denote by $P_t^{\alpha,x}$ the transition semigroup associated with system \eqref{smf4}. Namely, for every Borel bounded function $\varphi:\mathbb{R}^d\times \mathbb{R}^n\to \mathbb{R}$
\[P_t^{\alpha,x}\varphi(u,z)=\mathbb{E}_{(u,z)}\,\varphi(\Lambda^{x}(t)),\ \ \ \ \ \ t\geq 0,\ \ \ \ \ (u,z) \in\,\mathbb{R}^n\times \mathbb{R}^d,\]
where $\Lambda^{x}(t)$ is the solution of system \eqref{smf4}.
In view of \eqref{smf12} and the fact that \[\inf \left\{\mathfrak{Re}\,\la,\ \lambda \in\,\text{spec}(A)\right\}=:\bar{\la}>0,\]
we have that for every Lipschitz continuous mapping $\varphi:\mathbb{R}^d\times \mathbb{R}^n\to \mathbb{R}$
\begin{equation}
	\label{smf13}
	\left\vert P_t^{\alpha,x}(u,z)-\int_{\mathbb{R}^d\times \mathbb{R}^n} \varphi(u^\prime,z^\prime)\,d\mu_{\alpha,x}(u^\prime,z^\prime)\right\vert \leq c\,e^{-\omega_\alpha t}\left(1+\vert u\vert+\vert z\vert\right)[\varphi]_{\text{\tiny{Lip}}}, 
\end{equation}
with
\begin{equation}
\label{smf14}
\omega_\alpha:=\min\left\{ \frac {\gamma_0} \alpha, \bar{\lambda}\right\}.	
\end{equation}
In particular, this implies that for every $T>0$ and $(u,z) \in\,\mathbb{R}^d\times \mathbb{R}^n$,
\begin{equation}
\label{smf15}
\sup_{t\geq 0,\, x \in\,\mathbb{R}^d}\mathbb{E}_{(u,z)}\left|\frac 1T\int_{t}^{T+t} \varphi(\Lambda^x(s))\,ds-\int_{\mathbb{R}^d\times \mathbb{R}^n} \varphi(u^\prime,z^\prime)\,d\mu_{\alpha,x}(u^\prime,z^\prime)\right\vert\leq \frac{1}{\sqrt{\omega_\alpha T}}\left(1+|u|+|z|\right)c_\varphi.	
\end{equation}
For a proof of this last statement, see e.g. \cite{cerrai-09}.

\subsection{Uniform bounds} \label{ss3.2}
According to \eqref{system1-bis}, we have
\begin{equation}\label{last2}\begin{array}{l}\ds{d x_\e(t)=(\gamma^{-1}\,b)(x_\epsilon(t))\,dt-\mu(\e)\,\gamma^{-1}(x_\e(t))\,du_\epsilon(t)+(\gamma^{-1}\sigma)(x_\epsilon(t))z_\e(t)dt}\\[14pt]
\ds{\quad \quad \quad \quad \quad -\mu(\e)\gamma^{-1}(x_\epsilon(t))\left(S_1(x_\epsilon(t))+S_2(x_\epsilon(t))\,u_\epsilon(t)\right)\,dt.}\end{array}\end{equation}
Now,
\begin{equation}
\label{last1}
\gamma^{-1}(x_\e(t))\,du_\epsilon(t)=d(\gamma^{-1}(x_\e)\,u_\epsilon)(t)-D\gamma^{-1}(x_\e(t))[u_\e(t)+(\gamma^{-1}b)(x_\epsilon(t))]u_\epsilon(t)\,dt,	
\end{equation}
and
\[\begin{array}{l}
\ds{(\gamma^{-1}\sigma)(x_\epsilon(t))z_\e(t)dt=-\epsilon\,(\gamma^{-1}\sigma)(x_\epsilon(t))A^{-1}dz_\e(t)+(\gamma^{-1}\sigma)(x_\epsilon(t))A^{-1}Bdw(t)}\\[14pt]
\ds{\quad \quad =-\e\,d\left((\gamma^{-1}\sigma)(x_\epsilon)A^{-1}z_\e\right)(t)+\e\,D(\gamma^{-1}\sigma)(x_\e(t))[u_\e(t)+(\gamma^{-1}b)(x_\epsilon(t))]A^{-1}z_\e(t)\,dt}\\[14pt]
\ds{\quad \quad \quad \quad \quad \quad \quad \quad +(\gamma^{-1}\sigma)(x_\epsilon(t))A^{-1}Bdw(t).}	
\end{array}\]
Thus, if we replace this this identity and identity \eqref{last1} into \eqref{last2}, we obtain
\begin{equation}\label{smf63}\begin{array}{l}
\ds{x_{\epsilon}(t)=\int_0^t (\gamma^{-1}b)(x_{\epsilon}(r))\,dr+\int_0^t(\gamma^{-1}\sigma)(x_{\epsilon}(r))A^{-1}Bdw(r)}\\[14pt]
\ds{\quad +\mu(\epsilon)\int_0^t	[D\gamma^{-1}(x_{\epsilon}(r))\,u_{\epsilon}(r)]\,u_{\epsilon}(r)\,dr+\epsilon\int_0^t [D(\gamma^{-1}\sigma)(x_{\epsilon}(r))\,u_{\epsilon}(r)]\,A^{-1}z_\e(r)\,dr}\\[14pt]
\ds{\quad \quad \quad \quad \quad \quad \quad \quad \quad \quad \quad \quad+\mathcal{R}^1_{\epsilon}(t)+\int_0^t\mathcal{R}^2_{\epsilon}(r)\,dr,}
\end{array}\end{equation}
where
\begin{equation}\label{smf69}\begin{array}{l}
\ds{\mathcal{R}^1_{\epsilon}(t):=-\mu(\epsilon)\,\gamma^{-1}(x_\e(t))\,u_\epsilon(t)-\epsilon\,(\gamma^{-1}\sigma)(x_\epsilon(t))A^{-1}z_\e(t),}
\end{array}\end{equation}
and
\begin{equation}\label{smf70}\begin{array}{l}\ds{\mathcal{R}^2_{\epsilon}(r):=-\mu(\epsilon)\, [D\gamma^{-1}(x_{\epsilon}(r))\,\gamma^{-1}(x_{\epsilon}(r))b(x_{\epsilon}(r))]\,u_{\epsilon}(r)}\\[14pt]
\ds{\quad \quad \quad \quad -\epsilon\, [D(\gamma^{-1}\sigma)(x_{\epsilon}(r))\,\gamma^{-1}(x_{\epsilon}(r))b(x_{\epsilon}(r))]\,A^{-1}z_\e(r)}\\[14pt]
\ds{\quad \quad\quad \quad\quad \quad \quad \quad-\mu(\epsilon)\,\sqrt{\e}\,\gamma^{-1}(x_{\epsilon}(r))\left( S_1(x_{\e}(t))+ S_2(x_{\e}(t))u_{\epsilon}(t)\right).}\\[14pt]
\end{array}\end{equation}

\begin{Lemma}\label{lemma4.5}
Under Hypothesis \ref{hyp}, for every $p\geq 1$ and $T>0$ we have
\begin{equation}
\label{smf31}
\sup_{\epsilon \in\,(	0,1)}\,\sup_{t \in\,[0,T]}\mathbb{E}\,\vert \sqrt{\e}\,v_{\e}(t)\vert^{2p}\leq c_{p,T}\left(\vert x_0\vert^{2p}+\vert v_0\vert^{2p}+1\right).
	\end{equation}
	and 
\begin{equation}
\label{smf25}
\sup_{\epsilon \in\,(	0,1)}\,\sup_{t \in\,[0,T]}\mathbb{E}\,\vert \sqrt{\e}\,u_{\e}(t)\vert^{2p}\leq c_{p,T}\left(\vert x_0\vert^{2p}+\vert v_0\vert^{2p}+1\right).
	\end{equation}
\end{Lemma}

\begin{proof}

We have 
\begin{equation}\label{smf44}\begin{array}{l}
\ds{\frac 1{2p} \frac{d}{dt}	\vert v_{\epsilon}(t)\vert^{2p}=-\frac 1{\mu(\epsilon)} \vert v_{\epsilon}(t)\vert^{2(p-1)}\langle \gamma(x_{\epsilon}(t))v_{\epsilon}(t),v_{\epsilon}(t)\rangle}\\[14pt]
\ds{\quad \quad \quad +\frac 1{\mu(\epsilon)} \vert v_{\epsilon}(t)\vert^{2(p-1)}\Big(\langle b(x_{\epsilon}(t)),v_{\epsilon}(t)\rangle+\langle \sigma(x_{\epsilon}(t))z_\e(t),v_{\epsilon}(t)\rangle\Big)}\\[14pt]
\ds{\quad \quad \quad\quad \quad \quad\leq -\frac {\gamma_0}{2p\,\mu(\epsilon)}\vert v_{\epsilon}(t)\vert^{2p}+\frac{c_p}{\mu(\epsilon)}\left(1+\vert x_{\epsilon}(t)\vert^{2p}\right)+\frac{c_p}{\mu(\epsilon)\,\epsilon^p}\vert \sqrt{\epsilon}\,z_\epsilon(t)\vert^{2p}.}
\end{array}\end{equation}
In particular, this implies
\[\begin{array}{l}
\ds{\mathbb{E}\,\vert v_{\epsilon}(t)\vert^{2p}\leq e^{-\frac{\gamma_0}{\mu(\epsilon)}t}\vert v_0\vert^{2p}+\frac{c_p}{\mu(\epsilon)}\int_0^{t} e^{-\frac{\gamma_0}{\mu(\epsilon)} (t-s)}\left(1+\mathbb{E}\,\vert x_{\epsilon}(s)\vert^{2p}\right)\,ds}\\[14pt]
\ds{\quad \quad \quad \quad \quad \quad +\frac{c_p}{\mu(\epsilon)\,\epsilon^p}\int_0^{t} e^{-\frac{\gamma_0}{\mu(\epsilon)} (t-s)}\mathbb{E}\,\vert \sqrt{\epsilon}\,z_\epsilon(s)\vert^{2p}\,ds,}
\end{array}
\]
and thanks to \eqref{smf26} we conclude
\[\mathbb{E}\,\vert v_{\epsilon}(t)\vert^{2p}\leq \vert v_0\vert^{2p}+c_p\,\left(1+\sup_{s \in\,[0,t]}\mathbb{E}\,\vert x_{\epsilon}(s)\vert^{2p}\right)+\frac{c_{p,T}}{\epsilon^p}.\]
Moreover, since
\begin{equation}\label{smf30}\sup_{s \in\,[0,t]}\mathbb{E}\,\vert x_{\epsilon}(s)\vert^{2p}\leq c_p\,\vert x_0\vert^{2p}+c_{p,T}\int_{0}^t \mathbb{E}\,\vert v_{\epsilon}(s)\vert^{2p}\,dr,\end{equation}
we obtain
\[\mathbb{E}\,\vert v_{\epsilon}(t)\vert^{2p}\leq c_p\left(\vert x_0\vert^{2p}+\vert v_0\vert^{2p}+1\right)+c_{p,T}\,\int_{0}^t \vert v_{\epsilon}(s)\vert^{2p}\,dr+\frac{c_{p,T}}{\epsilon^p},\]
and  Gronwall's lemma gives \eqref{smf31}.

Now, recalling that
\[u_{\epsilon
}(t)=v_{\epsilon}(t)-(\gamma^{-1}b)(x_{\epsilon}(t)),\]
\eqref{smf31} implies
\[\mathbb{E}\,\vert u_{\epsilon}(t)\vert^{2p}\leq c_p(\vert v_0\vert^{2p}+ \vert x_0\vert^{2p}+1)+c_p\,\left(1+\sup_{s \in\,[0,t]}\mathbb{E}\,\vert x_{\epsilon}(s)\vert^{2p}\right)+\frac{c_{T,p}}{\epsilon^p}.\]
Therefore, combining this with \eqref{smf30} and \eqref{smf31}, we obtain  \eqref{smf25}. 
\end{proof}

\begin{Remark}
{\em  By proceeding as in the proof of Lemma \ref{lemma4.5} above, due to \eqref{smf44} we have
\[\vert v_{\epsilon}(t)\vert^{2p}\leq \vert v_0\vert^{2p}+c_p\,\left(1+\sup_{s \in\,[0,t]}\vert x_{\epsilon}(s)\vert^{2p}\right)+c_{p}\sup_{s \in\,[0,t]}\vert \sqrt{\epsilon}\,z_\epsilon(s)\vert^{2p}\,\e^{-p}.\]
Now, iIt is possible to show that for every $T>0$, $p\geq 1$  and $\delta>0$ there exists some constant $c_{p, \delta, T}>0$ such that
\begin{equation}
\label{smf43}
\mathbb{E}\,\sup_{t \in\,[0,T]}\vert \sqrt{\epsilon}\,z_\epsilon(t)\vert^p\leq c_{p, \delta, T}\,\epsilon^{-\delta}.	
\end{equation}
In particular,  for every $\delta>0$ we get
\begin{equation}\label{smf50}\mathbb{E}\sup_{s \in\,[0,t]}\vert v_{\epsilon}(t)\vert^{2p}\leq \vert v_0\vert^{2p}+c_p\,\left(1+\mathbb{E}\sup_{s \in\,[0,t]}\vert x_{\epsilon}(s)\vert^{2p}\right)+c_{p, \delta, T}\,\e^{-(p+\delta)}.\end{equation}
Thanks to \eqref{smf30}, by proceeding as in the proof of Lemma \ref{lemma4.5}, this allows to conclude that for every $\delta>0$ and $\epsilon \in\,(0,1)$
\begin{equation}
\label{smf45-bis}
\mathbb{E}\sup_{t \in\,[0,T]}\vert \sqrt{\epsilon}\,v_{\e}(t)\vert^{2p}\leq c_{p, \delta, T}\left(\vert x_0\vert^{2p}+\vert v_0\vert^{2p}+1\right)\,\epsilon^{-\delta},
\end{equation}
and 
\begin{equation}
\label{smf45}
\mathbb{E}\sup_{t \in\,[0,T]}\vert \sqrt{\epsilon}\,u_{\e}(t)\vert^{2p}\leq c_{p, \delta, T}\left(\vert x_0\vert^{2p}+\vert v_0\vert^{2p}+1\right)\,\epsilon^{-\delta}.	
\end{equation}
	
}
\hfill{$\Box$}	
\end{Remark}

All what we have seen above gives uniform bounds for $\sqrt{\e}\, v_{\epsilon,\mu}$ and $\sqrt{\epsilon}\, u_{\epsilon,\mu}$. If $\mu(\e)/\epsilon\to \alpha \in\,[0,+\infty)$, then the same bounds hold for $\sqrt{\mu(\epsilon)}\, v_{\epsilon}$ and $\sqrt{\mu(\epsilon)}\,u_{\epsilon}$ as well. However, when $\alpha=0$ we need a new proof for those bounds.
\begin{Lemma}
Under Hypothesis \ref{hyp}, for every $p\geq 1$  we have
\begin{equation}
\label{smf33}
\sup_{\epsilon\in\,(	0,1)}\,\sup_{t \in\,[0,T]}\mathbb{E}\,\vert \sqrt{\mu(\epsilon)}\,u_{\e}(t)\vert^{2p}\leq c_{p,T},  \ \ \ \ \ \    \sup_{\epsilon\in\,(	0,1)}\,\sup_{t \in\,[0,T]}\mathbb{E}\,\vert \sqrt{\mu(\epsilon)}\,u_{\e}(t)\vert^{2p}\leq c_{p,T}.
	\end{equation}
	Moreover, for every $\delta>0$ and $\epsilon \in\,(0,1)$  we have
	\begin{equation}
\label{smf46}
\mathbb{E}\sup_{t \in\,[0,T]}\,\vert \sqrt{\mu(\epsilon)}\,u_{\e}(t)\vert^{2}\leq c_{\delta, T}\,\mu(\e)^{-\delta},\ \ \ \ \ \ \mathbb{E}\sup_{t \in\,[0,T]}\,\vert \sqrt{\mu(\epsilon)}\,u_{\e}(t)\vert^{2}\leq c_{\delta, T}\,\mu(\e)^{-\delta}.	
\end{equation}

\end{Lemma}

\begin{proof}
As we have seen above, it remains to prove \eqref{smf33} and \eqref{smf46} under the assumption
\[\frac {\mu(\e)}{\epsilon}\to \infty.\]

{\em Step 1. Proof of \eqref{smf33} - }We have
\[\begin{array}{l}
\ds{d(\sqrt{\mu(\epsilon)}\,v_{\epsilon})(t)=-\frac{1}{\mu(\epsilon)}\,\gamma(x_{\epsilon}(t))\,\sqrt{\mu(\epsilon)}\,v_{\epsilon}(t)\,dt+\frac{1}{\sqrt{\mu(\epsilon)}}\,b(x_{\epsilon}(t))\,dt}\\[14pt]
\ds{\quad \quad \quad \quad \quad \quad \quad \quad \quad +\frac 1{\sqrt{\mu(\epsilon)}}\,\sigma(x_{\epsilon}(t))\,A^{-1}B\,dw(t)-\sqrt{\frac{\e}{\mu(\epsilon)}}\,\sigma(x_{\epsilon}(t))A^{-1}d (\sqrt{\e}\,z_\e)(t).}	
\end{array}
 \]
Thus, since 
\[\sigma(x_{\epsilon}(t))A^{-1}d(\sqrt{\e}\,z_\e)(t)=d(\sigma(x_{\epsilon})A^{-1}(\sqrt{\e}\,z_\e))(t)-D\sigma(x_{\epsilon}(t))v_{\epsilon}(t)A^{-1}(\sqrt{\e}z_\e)(t)\,dt,\]
if we define
\begin{equation}\label{smf61}\Lambda_{\e}(t):=\sqrt{\mu(\epsilon)}\, v_{\epsilon}(t)+\sqrt{\frac{\e}{\mu(\epsilon)}}\sigma(x_{\epsilon}(t))A^{-1}(\sqrt{\e}\,z_\e)(t),\end{equation}
we get
\[\begin{array}{l}
\ds{d\Lambda_{\e}(t)=\frac 1{\mu(\epsilon)}\,\Big[-\,\gamma(x_{\epsilon}(t))\Lambda_{\e}(t)+\sqrt{\frac{\e}{\mu(\epsilon)}}\,\gamma(x_{\epsilon}(t))\sigma(x_{\epsilon}(t))A^{-1}(\sqrt{\e}z_\e)(t)}\\[14pt]
\ds{+\sqrt{\mu(\epsilon)}\,b(x_{\epsilon}(t))+\sqrt{\e}\,D\sigma(x_{\epsilon}(t))(\sqrt{\mu(\epsilon)}\,v_{\epsilon})(t)A^{-1}(\sqrt{\e}z_\e)(t)\Big]\,dt+\frac 1{\sqrt{\mu(\epsilon)}}\,\sigma(x_{\epsilon}(t))\,A^{-1}B\,dw(t).}	
\end{array}\]
Thanks to It\^o's formula, we get
\[\begin{array}{l}
\ds{\frac 1{2p}d\vert \Lambda_{\e}(t)\vert ^{2p}=\frac{1}{\mu(\epsilon)}\vert \Lambda_{\e}(t)\vert ^{2(p-1)}\Big(-\langle\gamma(x_{\epsilon}(t))\Lambda_{\e}(t),\Lambda_{\e}(t)\rangle}\\[14pt]
\ds{\quad +\sqrt{\frac{\e}{\mu(\epsilon)}}\,\langle\gamma(x_{\epsilon}(t))\sigma(x_{\epsilon}(t))A^{-1}(\sqrt{\e}z_\e)(t),\Lambda_{\e,\mu}(t)\rangle+\sqrt{\mu(\epsilon)}\,\langle b(x_{\epsilon}(t)),\Lambda_{\e}(t)\rangle}\\[14pt]
\ds{\quad \quad \quad \quad +\sqrt{\e}\,\langle D\sigma(x_{\epsilon}(t))(\sqrt{\mu(\epsilon)}\,v_{\epsilon})(t)A^{-1}(\sqrt{\e}z_\e)(t),\Lambda_{\e}(t)\rangle}\\[14pt]
\ds{\quad +\frac 1{2} \text{Tr}\Big[(\sigma^\star\sigma)(x_{\epsilon}(t))\,(A^{-1}B)(A^{-1}B)^\star(I+2(p-1)\frac{\Lambda_{\e}(t)\otimes \Lambda_{\e}(t)}{\vert \Lambda_{\e}(t)\vert^2}\Big]\Big)\,dt}\\[14pt]
\ds{\quad \quad \quad \quad \quad +\frac 1{\sqrt{\mu(\epsilon)}}\,\vert \Lambda_{\e}(t)\vert ^{2(p-1)}\,\langle \sigma(x_{\epsilon}(t))\,A^{-1}B\,dw(t),\Lambda_{\e}(t)\rangle.}	
\end{array}\]
Thus, 
\begin{equation}\label{smf36}\begin{array}{l}
\ds{\frac 1{2p}d\vert \Lambda_{\e}(t)\vert ^{2p}\leq -\frac{\gamma_0}{p\mu(\epsilon)}\,\vert\Lambda_{\e}(t)\vert^{2p}\,dt+\frac{c_p}{\mu(\epsilon)}\,\left(\frac{\e}{\mu(\epsilon)}\right)^p\,\vert \sqrt{\e}z_\e(t)\vert^{2p}\,dt}\\[14pt]
\ds{\quad \quad \quad +\frac{c_p}{\mu(\epsilon)}\left(1+\vert \sqrt{\mu(\epsilon)}\,x_{\epsilon}(t)\vert^{2p}\right)\,dt+\bar{c}_p\,\frac{\epsilon^{\,p}}{\mu(\epsilon)}\,\vert \sqrt{\mu(\epsilon)}\,v_{\epsilon}(t)\vert^{2p} \,\vert\sqrt{\e}z_\e(t)\vert^{2p}\,dt}\\[14pt]
\ds{\quad \quad \quad \quad \quad \quad +\frac 1{\sqrt{\mu(\epsilon)}}\,\vert \Lambda_{\e}(t)\vert ^{2(p-1)}\,\langle \sigma(x_{\epsilon}(t))\,A^{-1}B\,dw(t),\Lambda_{\e}(t)\rangle.}\end{array}\end{equation}
Now, recalling how $\Lambda_{\e}(t)$ is defined in \eqref{smf61}, we have
\begin{equation}\label{smf34}\begin{array}{l}
\ds{\bar{c}_p\,\frac{\e^{\,p}}{\mu(\epsilon)}\,\vert \sqrt{\mu(\epsilon)}\,v_{\epsilon}(t)\vert^{2p} \,\vert\sqrt{\e}z_\e(t)\vert^{2p}\leq \hat{c}_p\,\frac{\e^{\,p}}{\mu(\epsilon)}\,\vert \Lambda_{\epsilon}(t)\vert^{2p} \,\vert\sqrt{\e}z_\e(t)\vert^{2p}+c_p\,\frac{\e^{2p}}{\mu(\e)^{p+1}}\,\vert \sqrt{\e}z_\e(t)\vert^{4p}}\\[14pt]
\ds{\quad \quad \leq \frac{\gamma_0}{2p\,\mu(\epsilon)}\,\vert\Lambda_{\e}(t)\vert^{2p}+\hat{c}_p\,\frac{\e^{\,p}}{\mu(\epsilon)}\,\vert \Lambda_{\epsilon}(t)\vert^{2p} \,\vert\sqrt{\e}z_\e(t)\vert^{2p}\,\mathbb{I}_{D_{\e}(t)}+c_p\,\frac{\e^{2p}}{\mu(\e)^{p+1}}\,\vert \sqrt{\e}z_\e(t)\vert^{4p},}
\end{array}\end{equation}
where
\[D_{\e}(t):=\left\{\hat{c}_p\,\e^{\,p}\,\vert\sqrt{\e}z_\e(t)\vert^{2p}\geq \frac{ \gamma_0}{2p}\right\}.\]
Notice that, thanks to \eqref{smf26}, for every $k> 0$ there exists $c_k>0$ such that
\begin{equation}
\label{smf40}
\sup_{t \in\,[0,T]}\,\mathbb{P}(D_{\e}(t))\leq c_k \epsilon^k.
\end{equation}
In particular, if we replace \eqref{smf34} into \eqref{smf36}, we obtain
\begin{equation}\label{smf37}\begin{array}{l}
\ds{d\vert \Lambda_{\e}(t)\vert ^{2p}\leq -\frac{\gamma_0}{\mu(\epsilon)}\,\vert\Lambda_{\e}(t)\vert^{2p}\,dt+\frac{c_p}{\mu(\epsilon)}\,\left(\frac{\e}{\mu(\epsilon)}\right)^p\,\vert \sqrt{\e}z_\e(t)\vert^{2p}\,dt+\frac{c_p}{\mu(\epsilon)}\left(1+\vert \sqrt{\mu(\epsilon)}\,x_{\epsilon}(t)\vert^{2p}\right)\,dt}\\[14pt]

\ds{\quad \quad +\hat{c}_p\,\frac{\e^{\,p}}{\mu(\epsilon)}\,\vert \Lambda_{\epsilon}(t)\vert^{2p} \,\vert\sqrt{\e}z_\e(t)\vert^{2p}\,\mathbb{I}_{D_{\e}(t)}\,dt+c_p\,\frac{\e^{2p}}{\mu(\e)^{p+1}}\,\vert \sqrt{\e}z_\e(t)\vert^{4p}\,dt}\\[14pt]
\ds{\quad \quad \quad \quad \quad \quad \quad \quad +\frac 1{\sqrt{\mu(\epsilon)}}\,\vert \Lambda_{\e}(t)\vert ^{2(p-1)}\,\langle \sigma(x_{\epsilon}(t))\,A^{-1}B\,dw(t),\Lambda_{\e}(t)\rangle,}\\[14pt]
\ds{\quad }	
\end{array}\end{equation}
and this yields
\begin{equation}\label{smf38}\begin{array}{l}
\ds{\mathbb{E}\vert \Lambda_{\e}(t)\vert ^{2p}\leq e^{-\frac{\gamma_0}{\mu(\epsilon)}t}\,\mu(\e)^p\vert v_0\vert^{2p} +\frac{c_p}{\mu(\epsilon)}\int_0^t e^{-\frac{ \gamma_0}{\mu(\epsilon)}(t-s)}\,\left(1+\left(\frac{\e}{\mu(\epsilon)}\right)^{2p}\,\mathbb{E}\vert \sqrt{\e}z_\e(s)\vert^{4p}\right)\,ds}\\[14pt]
\ds{\quad \quad \quad \quad  +\frac{c_p}{\mu(\epsilon)}\,\int_0^te^{-\frac{\gamma_0}{\mu(\epsilon)}(t-s)}\mathbb{E}\vert \sqrt{\mu(\epsilon)}\,x_{\epsilon}(s)\vert^{2p}\,ds}\\[14pt]
\ds{\quad \quad \quad \quad \quad \quad \quad \quad  +\hat{c}_p\,\frac{\e^p}{\mu(\epsilon)}\int_0^te^{-\frac{\gamma_0}{\mu(\epsilon)}(t-s)}\,\mathbb{E}\left(\vert\sqrt{\e}z_\e(s)\vert^{2p}\,\vert\Lambda_{\e}(s)\vert^{2p}\,\mathbb{I}_{D_{\e}(s)}\right)\,ds.}
\end{array}\end{equation}
According to \eqref{smf26} and  \eqref{smf40}, with $k$ sufficiently large, depending on $p$, we have
\[\begin{array}{l}
\ds{\mathbb{E}\left(\vert\sqrt{\e}z_\e(s)\vert^{2p}\,\vert\Lambda_{\e}(s)\vert^{2p}\,\mathbb{I}_{D_{\e}(s)}\right)\leq \left(\mathbb{E}\vert\sqrt{\e}z_\e(s)\vert^{8p}\right)^{1/4}\left(\mathbb{E}\vert\Lambda_{\e}(s)\vert^{4p}\right)^{1/2}\left(\mathbb{P}(D_{\e}(s))\right)^{1/4}}\\[14pt]
\ds{\quad \quad \leq c_p\, \epsilon^{\,p}	\left(\left(\frac {\mu(\epsilon)} \epsilon\right)^{2p}\mathbb{E}\vert \sqrt{\epsilon}\,v_{\e}(s)\vert^{4p}+\left(\frac{\e}{\mu(\epsilon)}\right)^{2p}\mathbb{E}\vert \sqrt{\e}z_\e(s)\vert^{4p} \right)^{1/2}.}
\end{array}\]
Then, thanks again to \eqref{smf26} and to \eqref{smf31}, since we are assuming $\e/\mu(\epsilon)\to 0$, we conclude
\[\mathbb{E}\left(\vert\sqrt{\e}z_\e(s)\vert^{2p}\,\vert\Lambda_{\e}(s)\vert^{2p}\,\mathbb{I}_{D_{\e}(s)}\right)\leq c_{p,T}\, \epsilon^{p}\left(\frac{\mu(\epsilon)}{\e}+\frac{\e}{\mu(\epsilon)}\right)^p\leq c_{p,T},
\]
and if we plug the inequality above into \eqref{smf38}, thanks again to \eqref{smf26},
we obtain
\begin{equation}\label{smf42}\begin{array}{l}
\ds{\mathbb{E}\vert \Lambda_{\e}(t)\vert ^{2p}\leq \vert\sqrt{\mu(\epsilon)}\, v_0\vert^{2p} +c_{p,T}\,\left(1+\sup_{s \in\,[0,t]}\mathbb{E}\vert \sqrt{\mu(\epsilon)}\,x_{\epsilon}(s)\vert^{2p}\right)+c_{p,T}.}\end{array}\end{equation}
In particular, since
\[\vert \sqrt{\mu(\epsilon)}\,v_{\epsilon}(t)\vert^{2p}\leq c_p\left(\vert \Lambda_{\e}(t)\vert ^{2p}+\vert \sqrt{\e} z_\e(t)\vert^{2p}\right),\]
from \eqref{smf26}, \eqref{smf30} and \eqref{smf42} we obtain 
\[\mathbb{E}\vert \sqrt{\mu(\epsilon)}\,v_{\e}(t)\vert ^{2p}\leq c_{p,T}\,\left(1+\sup_{s \in\,[0,t]}\mathbb{E}\vert \sqrt{\mu(\epsilon)}\,x_{\epsilon}(s)\vert^{2p}\right)\leq c_{p,T} \left(1+\int_0^t\mathbb{E}\vert \sqrt{\mu(\epsilon)}\,v_{\epsilon}(s)\vert^{2p}\,ds\right).\] 
The Gronwall's lemma implies
\[\mathbb{E}\vert \sqrt{\mu(\epsilon)}\,v_{\e}(t)\vert ^{2p}\leq c_{p,T},\]
and then, by proceeding as at the end of Lemma \ref{lemma4.5}, we get \eqref{smf33} for $u_{\e,\mu}$.

{\em Step 2. Proof of \eqref{smf46} - } Thanks to \eqref{smf37}  we have

\begin{equation}\label{smf48}\begin{array}{l}
\ds{\sup_{s \in\,[0,t]}\vert \Lambda_{\e}(t)\vert ^2\leq \vert\Lambda_{\e}(0)\vert^2 +c\left(1+\frac{\e}{\mu(\epsilon)}\,\sup_{s \in\,[0,t]}\vert \sqrt{\e}z_\e(s)\vert^2+\frac{\e^2}{\mu(\epsilon)}\,\sup_{s \in\,[0,t]}\vert	\sqrt{\e}z_\e(s)\vert^4\right)}\\[14pt]
\ds{ \quad \quad +c\left(1+\sup_{s \in\,[0,t]}\vert \sqrt{\mu(\epsilon)}\,x_{\epsilon}(s)\vert^2\right) +\bar{c}\,\epsilon\,\sup_{s \in\,[0,t]}\left(\vert\sqrt{\e}z_\e(s)\vert^2\,\vert\Lambda_{\e}(s)\vert^2\,\mathbb{I}_{D_{\e}(s)}\right)}\\[15pt]
\ds{\quad \quad \quad \quad +\frac 1{\sqrt{\mu(\epsilon)}}\,\sup_{s \in\,[0,t]}\left|\int_0^s e^{-\frac{\gamma_0}{\mu(\epsilon)}(s-r)}\langle \sigma(x_{\epsilon}(r))\,A^{-1}B\,dw(r),\Lambda_{\e}(r)\rangle\right|}
\end{array}\end{equation}

We have
\[\begin{array}{l}
\ds{\mathbb{E}\,\sup_{s \in\,[0,t]}\left(\vert\sqrt{\e}z_\e(s)\vert^2\,\vert\Lambda_{\e}(s)\vert^2\,\mathbb{I}_{D_{\e}(s)}\right)}\\[14pt]
\ds{\quad \quad \quad \leq \left(\mathbb{E}\sup_{s \in\,[0,t]}\vert\sqrt{\e}z_\e(s)\vert^8\right)^{1/4}\left(\mathbb{E}\sup_{s \in\,[0,t]}\vert\Lambda_{\e}(s)\vert^4\right)^{1/2}\left(\mathbb{P}(D^\star_{\e}(t)\right)^{1/4},}	
\end{array}\]
where
\[D^\star_{\e}(t):=\left\{\bar{c}\,\sqrt{\e}\,\sup_{s \in\,[0,t]}\vert\sqrt{\e}z_\e(s)\vert\geq \frac{\gamma_0}{2}\right\}.\]
According to \eqref{smf43}
we have that for every $k\geq 1$ there exists $c_{k,t}>0$ such that 
\[\mathbb{P}(D^\star_{\e}(t))\leq c_{k,t}\, \e^k.\]
Thus, thanks again to \eqref{smf43} and to  \eqref{smf50} and \eqref{smf45-bis}, by taking $k$ large enough in the inequality above, we conclude
\begin{equation}\label{smf51}\sup_{\e\in\,(0,1)}\mathbb{E}\,\sup_{s \in\,[0,t]}\left(\vert\sqrt{\e}z_\e(s)\vert^2\,\vert\Lambda_{\e}(s)\vert^2\,\mathbb{I}_{D_{\e}(s)}\right)\leq c_T,\ \ \ \ t \in\,[0,T].\end{equation}

Next, for every $\beta \in\,(0,1/2)$ we have 	
\[\begin{array}{l}
\ds{J_{\epsilon}(s):=\int_0^s e^{-\frac{\gamma_0}{\mu(\epsilon)}(s-r)}\langle \sigma(x_{\epsilon}(r))\,A^{-1}B\,dw(r),\Lambda_{\e}(r)\rangle
=c_\beta \int_0^s e^{-\frac{\gamma_0}{\mu(\epsilon)}(s-r)}(s-r)^{\beta-1}Y_\beta(r)\,dr,
}
\end{array}\]
where
\[Y_\beta(r):=\int_0^r e^{-\frac{\gamma_0}{\mu(\epsilon)}(r-\rho)}(r-\rho)^{-\beta}\langle \sigma(x_{\epsilon}(\rho))\,A^{-1}B\,dw(\rho),\Lambda_{\e}(\rho)\rangle.\]
Thus, for every $p>1/\beta$, we get
\[\begin{array}{l}
\ds{\vert J_{\epsilon}(s)\vert^{p}\leq c_{p, \beta}\left(\int_0^s	e^{-\frac{\gamma_0 p}{\mu(\epsilon)(p-1)}r}r^{\frac{(\beta-1) p}{p-1}}\,dr\right)^{p-1}}\int_0^s\vert Y_\beta(r)\vert^p\,dr \leq c_{p, \beta}\,\mu(\e)^{-\beta p-1}\int_0^s\vert Y_\alpha(r)\vert^p\,dr.
\end{array}\]
This implies that
\[\begin{array}{l}
\ds{\mathbb{E}\sup_{s \in\,[0,t]}\vert J_{\epsilon}(s)\vert^{p}\leq c_{p, \beta}\,\mu(\e)^{\beta p-1}\int_0^t\mathbb{E}\,\vert Y_\beta(r)\vert^p\,dr}\\[14pt]
\ds{\leq c_{p, \beta}\,\mu(\e)^{\beta p-1}\mathbb{E}\int_0^t\left(\int_0^r e^{-2\frac{\gamma_0}{\mu(\epsilon)}(r-\rho)}(r-\rho)^{-2\beta}\vert (A^{-1}B)^\star\sigma^\star(x_{\epsilon}(\rho))\,\Lambda_{\e}(\rho)\vert^2\,d\rho\right)^{p/2}\,dr}	\\[14pt]
\ds{\leq c_{p, \beta}\,\mu(\e)^{\beta p-1}\,\mu(\e)^{-\beta  p+\frac p2}}\int_0^t\mathbb{E}\vert\Lambda_{\e}(s)\vert^p\,ds= c_{p, \alpha}\,\mu(\e)^{\frac {p}2-1}\int_0^t\mathbb{E}\vert\Lambda_{\e}(s)\vert^p\,ds,
\end{array}\]
so that for every $p\geq 1$
\begin{equation}
\label{smf60}
\mathbb{E}\sup_{s \in\,[0,t]}\frac 1{\sqrt{\mu(\epsilon)}}\left|\int_0^s e^{-\frac{\gamma_0}{\mu(\epsilon)}(s-r)}\langle \sigma(x_{\epsilon}(r))\,A^{-1}B\,dw(r),\Lambda_{\e}(r)\rangle\right|\leq c_{p, \alpha}\,\mu(\e)^{-\frac{1}{p}}\left(\int_0^t\mathbb{E}\vert\Lambda_{\e}(s)\vert^p\,ds\right)^{1/p}.	\end{equation}
In particular, due to Step 1, we obtain that for every $\delta>0$ there exists some $c_{\delta, T}>0$ such that
\begin{equation}
\label{smf60-bis}
\mathbb{E}\sup_{s \in\,[0,t]}\frac 1{\sqrt{\mu(\epsilon)}}\left|\int_0^s e^{-\frac{\gamma_0}{\mu(\epsilon)}(s-r)}\langle \sigma(x_{\epsilon}(r))\,A^{-1}B\,dw(r),\Lambda_{\e,\mu(\epsilon)}(r)\rangle\right|\leq c_{\delta, T}\,\mu(\e)^{-\delta}.
\end{equation}

Finally, if we plug \eqref{smf50} and \eqref{smf60-bis} into \eqref{smf48}, thanks to \eqref{smf43} we conclude  
\begin{equation}\label{smf48-bis}\begin{array}{l}
\ds{\mathbb{E}\,\sup_{s \in\,[0,t]}\vert \Lambda_{\e}(t)\vert ^2\leq \vert\Lambda_{\e}(0)\vert^2 +c\,\mathbb{E}\,\sup_{s \in\,[0,t]}\vert \sqrt{\mu(\epsilon)}\,x_{\epsilon}(s)\vert^2+ c_{\delta, T} \mu(\e)^{-\delta}.}
\end{array}\end{equation}
This allows to obtain \eqref{smf46}, by proceeding as in Step 1.

\end{proof}

\begin{Remark}
{\em If we put together \eqref{smf25} and \eqref{smf33}, for every $p\geq 1$ we obtain
\begin{equation}
\label{smf66}
\sup_{\e \in\,(0,1)}\sup_{t \in\,[0,T]}\mathbb{E}\,\vert \sqrt{\mu(\e)+\e}\,v_{\epsilon}(t)\vert^{2p}	\leq c_{p,T},\ \ \ \ \ \ \sup_{\e \in\,(0,1)}\sup_{t \in\,[0,T]}\mathbb{E}\,\vert \sqrt{\mu(\e)+\e}\, u_{\epsilon}(t)\vert^{2p}	\leq c_{p,T},
\end{equation}
and
\begin{equation}
\label{smf66-bis}
\sup_{\e \in\,(0,1)}\mathbb{E}\,\sup_{t \in\,[0,T]}\vert \sqrt{\mu(\e)+\e}\,x_{\epsilon}(t)\vert^{2p}	\leq c_{p,T}.
\end{equation}}
\hfill{$\Box$}		
\end{Remark}

\subsection{Tightness} We are now going to prove that the family $\{\mathcal{L}(x_{\e})\}_{\epsilon \in\,(01)}$ is tight in $C([0,T];\mathbb{R}^d)$.
\begin{Lemma}
Let $\mathcal{R}^1_{\epsilon}(t)$ and $\mathcal{R}^2_{\epsilon}(t)$ be the mappings defined in \eqref{smf69} and \eqref{smf70}, respectively. For every $T>0$ and $p\geq 1$ and for every $\e \in\,(0,1)$  we have
\begin{equation}\label{smf72-bis}
\mathbb{E}\sup_{t \in\,[0,T]}	\vert \mathcal{R}^1_{\epsilon}(t)\vert^p \leq c_{p,T}\left(\mu(\e)+\epsilon\right)^{p/4}.	
\end{equation}
 \begin{equation}\label{smf75-bis}\sup_{\e \in\,(0,1)}\mathbb{E}\sup_{0\leq s<t\leq T}\Big|\frac 1{\sqrt{t-s}}\int_s^t\mathcal{R}^2_{\epsilon}(r)\,dr\Big|^p\leq c_{p,T}. \end{equation}
\end{Lemma}

\begin{proof}
Thanks to \eqref{smf26} and \eqref{smf33}, we have
\begin{equation}
\label{smf181}
\sup_{t \in\,[0,T]}\mathbb{E}\,\vert \mathcal{R}^1_{\epsilon}(t)\vert^p \leq c_{p,T}\left(\mu(\e)+\epsilon\right)^{\frac p2-\delta}.	
\end{equation}
Moreover
\[\sup_{t \in\,[0,T]}	\vert \mathcal{R}^1_{\epsilon}(t)\vert \leq c\Big(\sup_{t \in\,[0,T]}\mu(\epsilon) \,\vert u_{\epsilon}(t)\vert+\sup_{t \in\,[0,T]}\epsilon \,\vert z_{\epsilon}(t)\vert\Big).\]
Thus,  thanks to \eqref{smf43} and \eqref{smf46}, for every $p\geq 1$ and $\delta>0$  we get
\begin{equation}\label{smf72}
\mathbb{E}\sup_{t \in\,[0,T]}	\vert \mathcal{R}^1_{\epsilon}(t)\vert^p \leq c_{p,T}\left(\mu(\e)+\epsilon\right)^{p/4}.	
\end{equation}

As for $\mathcal{R}^2_{\epsilon}(r)$, due to \eqref{R(x,u)} we have
\begin{equation}\label{smf76}
\begin{array}{l}
\ds{\vert \mathcal{R}^2_{\epsilon}(r)\vert \leq c\, \sqrt{\mu(\e)}\left(1+\vert x_{\epsilon}(r)\vert \right)\,\vert \sqrt{\mu(\e)}\,u_{\epsilon}(r)\vert +c\, \sqrt{\epsilon}\left(1+\vert x_{\epsilon}(r)\vert \right)\,\vert \sqrt{\epsilon}\,z_{\epsilon}(r)\vert}\\[14pt]
\ds{\quad \quad\quad \quad\quad \quad \quad\quad \quad+\sqrt{\e}\,\left( 1+\vert \sqrt{\mu(\epsilon)}\,x_{\e}(r)\vert^2+ \vert \sqrt{\mu(\epsilon)}\,u_{\epsilon}(r)\vert^2\right).}
\end{array}\end{equation}
In view of \eqref{smf66} and \eqref{smf66-bis}, by H\"older inequality this implies  \eqref{smf75-bis}.
\end{proof}

\begin{Lemma}
For every $p\geq 1$ and $T>0$, we have	
\begin{equation}\label{smf80}
\sup_{\e\in\,(0,1)}\mathbb{E}\sup_{s \in\,[0,T]}\vert x_{\epsilon}(s)\vert^p \leq c_{p, T}.
\end{equation}
Moreover, the family $\{\mathcal{L}(x_{\e})\}_{\e \in\,(0,1)}$ is tight in $C([0,T];\mathbb{R}^d)$.
\end{Lemma}

\begin{proof}
According to identity \eqref{smf63},  we have
\begin{equation}\label{smf83}\begin{array}{l}
\ds{\vert (x_{\epsilon}(t)+\mathcal{R}^1_{\epsilon}(t)-(x_{\epsilon}(s)+\mathcal{R}^1_{\epsilon}(s))\vert^p}\\[14pt]
\ds{\leq c_p\int_s^t \Big(1+\vert x_{\epsilon}(r)\vert^p+\vert \sqrt{\mu(\epsilon)}\,u_{\epsilon}(r)\vert ^{2p}+\vert \sqrt{\e}\,u_{\epsilon}(r)\vert^p\,\vert\sqrt{\e}\,z_\e(r)\vert^p\Big)\,dr|t-s|^{p-1}}\\[14pt]
\ds{+c_p\Big|\int_s^t\mathcal{R}^2_{\epsilon}(r)\,dr\Big|^p+c_p \Big|\int_s^t\gamma^{-1}(x_{\epsilon}(r))\sigma(x_{\epsilon}(r))A^{-1}Bdw(r)\Big|^p.}
\end{array}\end{equation}
Hence, due to
estimates \eqref{smf26}, \eqref{smf33},   \eqref{smf72-bis} and \eqref{smf75-bis}, if we take $s=0$, we get
\[\begin{array}{l}
\ds{\mathbb{E}\sup_{t \in\,[0,T]}\vert x_{\epsilon}(t)\vert^p\leq c_{p,T}\int_0^T \mathbb{E}\,\vert x_{\epsilon}(r)\vert^p\,dr+c_{p,T}\Big(1+\left(\mu(\e)+\epsilon\right)^{p/2}\Big),}	
\end{array}\]
and Gronwall's lemma implies \eqref{smf80}.

Next, by using again estimates \eqref{smf26}, \eqref{smf33}, \eqref{smf72-bis}, \eqref{smf75-bis}, and \eqref{smf80}, for every $p\geq 1$ we get 
\begin{equation}
	\label{smf105}
	\begin{array}{l}
\ds{\mathbb{E}\,\big| \big(x_{\epsilon}(t)+\mathcal{R}^1_{\epsilon}(t)\big)
 - \big(x_{\epsilon}(s)+\mathcal{R}^1_{\epsilon}(s)\big) \big|^p}\\[14pt]
\ds{\quad \quad \quad \quad \leq c_p \int_s^t \Big(1+\mathbb{E}\,|x_{\epsilon}(r)|^p\Big)\,dr\, |t-s|^{p-1} +c_T\,|t-s|^{p/2}\leq c_{p,T}\,|t-s|^{p/2}.}
\end{array}
\end{equation}

As a consequence of Kolmogorov's criterion, this, together with \eqref{smf80}, implies that the family 
\[
\{\,\mathcal{L}( x_{\epsilon} + \mathcal{R}^1_{\epsilon} )\,\}_{\epsilon\in(0,1)}
\]
is tight in $C([0,T];\mathbb{R}^d)$. Hence, since by \eqref{smf72-bis}
\[
\lim_{\epsilon\to 0} \mathcal{R}^1_{\epsilon}=0, \quad \quad \text{in } L^p(\Omega;C([0,T];\mathbb{R}^d)),
\]
we obtain the tightness of the family $
\{\, \mathcal{L}(x_{\epsilon} )\,\}_{\epsilon\in(0,1)}$
 in $C([0,T];\mathbb{R}^d)$.
\end{proof}

\begin{Remark}
{\em   \begin{enumerate}
\item[1.]  As a consequence of 
\eqref{smf80}, in view of \eqref{smf76} and \eqref{smf46}, we have that
for every $\lambda>0$ and  $p\geq 1$
\begin{equation}
\label{smf90}
\mathbb{E}\sup_{|t-s|\leq \lambda}\Big|\int_s^t\mathcal{R}^2_{\epsilon}(r)\,dr\Big|^p	\leq c_{p,T}\,\lambda^{p/2}(\sqrt{\mu(\epsilon)}+\sqrt{\e})^{p}.
\end{equation}
\item[2.] In view of \eqref{smf181} and \eqref{smf105}, for every $0\leq s<t\leq T$ and $\e \in\,(0,1)$ we have
\begin{equation}
\label{smf106}
\mathbb{E}\,\big| x_{\epsilon}(t)
 - x_{\epsilon}(s)\big|^p \leq c_{p,T}\,|t-s|^{p/2}+	\left(\mu(\e)+\epsilon\right)^{p/2}.
\end{equation}
   	
   \end{enumerate}
 
}
\hfill{$\Box$}		
\end{Remark}

Since the family $
\{\, \mathcal{L}(x_{\epsilon} )\,\}_{\epsilon\in(0,1)}$
 is tight in $C([0,T];\mathbb{R}^d)$, in order to prove Theorem \ref{teo2} we need to identify  any weak limit point for $
\{\, x_{\epsilon} \,\}_{\epsilon\in(0,1)}$ in $C([0,T];\mathbb{R}^d)$ with  the solution of equation \eqref{main-1}. Then, since uniqueness holds for equation \eqref{main-1}, by a classical  argument we conclude that $
\{\, x_{\epsilon} \,\}_{\epsilon\in(0,1)}$ converges in probability to $x$ in the space $C([0,T];\mathbb{R}^d)$.

\section{Identification of the limit}

Let $\{\e_n\}_{n \in\,\mathbb{N}}$ be a sequence such that
\begin{equation}\label{smf105-bis}\lim_{n\to\infty }\e_n=0,\ \ \ \ \ \lim_{n\to\infty }\frac{\mu(\epsilon_n)}{\e_n}=\alpha \in\,[0,+\infty],\end{equation}
and assume $x_n(t):=x_{\e_n}(t)$ converges weakly in $C([0,T];\mathbb{R}^d)$ to some $C([0,T];\mathbb{R}^d)$-valued process $x(t)$. 
In view of Skhorohod's theorem,  we can assume that such convergence happens in $\mathbb{P}$-a.s. sense.

In what follows, our goal is showing that for every $\alpha \in\,[0,+\infty]$  the limiting process $x$ satisfies the equation
\begin{equation}
\label{smf120}
\begin{array}{l}
\ds{x(t)=x_0+\int_0^t (\gamma^{-1}b)(x(r))\,dr+\int_0^t f_\alpha(x(r))\,dr+ \int_0^t(\gamma^{-1}\sigma)(x(r))A^{-1}Bdw(r),}\end{array}\end{equation}where for every $\alpha \in\,[0,+\infty]$ the function $f_\alpha$ is defined in \eqref{smf91}. In our proof we will need to consider separately the cases $\alpha  \in\,(0,+\infty)$, $\alpha=0$ and $\alpha=\infty$.

\subsection{Case $\alpha \in\,(0,+\infty)$}

In \eqref{smf63}
we have shown that for every  $ t \in\,[0,T]$ and $n \in\,\mathbb{N}$
\begin{equation}\label{smf63-bis}\begin{array}{l}
\ds{x_{n}(t)=x_0+\int_0^t (\gamma^{-1}b)(x_{n}(r))\,dr+\int_0^t(\gamma^{-1}\sigma)(x_{n}(r))A^{-1}Bdw(r)}\\[14pt]
\ds{+\alpha_n\int_0^t	[D\gamma^{-1}(x_{n}(r))\,\tilde{u}_{n}(r)]\,\tilde{u}_{n}(r)\,dr+\int_0^t [D(\gamma^{-1}\sigma)(x_{n}(r))\,\tilde{u}_{n}(r)]\,A^{-1}\tilde{z}_n(r)\,dr}\\[14pt]
\ds{\quad \quad \quad \quad \quad \quad\quad \quad \quad \quad \quad \quad+\mathcal{R}^1_{n}(t)+\int_0^t\mathcal{R}^2_{n}(r)\,dr,}
\end{array}\end{equation}
where  $\mu_n:=\mu(\epsilon_n)$, $\alpha_n:=\mu_n/\e_n$ and 
\[\tilde{u}_n(t):=\sqrt{\e_n}\,u_{\e_n}(t),\ \ \ \ \ \tilde{z}_n(t):=\sqrt{\e_n}\,z_{\e_n}(t),\ \ \ \ \ \ \mathcal{R}^i_{n}(t)=\mathcal{R}^i_{\e_n}(t),\ \ \ i=1,2.\]
Thus,  for every $n \in\,\mathbb{N}$ we have
\begin{equation}\label{smf90}\begin{array}{l}
\ds{x_{n}(t)=x_0+\int_0^t (\gamma^{-1}b)(x_{n}(r))\,dr+\int_0^t f_\alpha(x_n(r))\,dr+ \int_0^t(\gamma^{-1}\sigma)(x_{n}(r))A^{-1}Bdw(r)}\\[14pt]
\ds{\quad \quad \quad \quad \quad \quad \quad \quad \quad \quad \quad  +\Theta^\alpha_n(t)+\mathcal{R}^1_{n}(t)+\int_0^t\mathcal{R}^2_{n}(r)\,dr,}\end{array}\end{equation}
  with
\begin{equation}\label{smf92}\begin{array}{l}
\ds{\Theta^\alpha _n(t):=
\alpha_n\int_0^t	[D\gamma^{-1}(x_{n}(r))\,\tilde{u}_{n}(r)]\,\tilde{u}_{n}(r)\,dr}\\[14pt]
\ds{\quad \quad \quad \quad \quad \quad +\int_0^t[D(\gamma^{-1}\sigma)(x_{n}(r))\,\tilde{u}_{n}(r)]\,A^{-1}\tilde{z}_n(r)\,dr-\int_0^t f_\alpha(x_n(r))\,dr .}
\end{array}\end{equation}
In particular, in view of
\eqref{smf72-bis} and \eqref{smf90}, Theorem \ref{teo2} follows once we prove \begin{equation}
\label{smf95}
\lim_{n\to \infty}\mathbb{E}\sup_{t \in\,[0,T]}\vert \Theta^\alpha_n(t)\vert=0.	
\end{equation}

If we define 
\[g^i_{l,j}(x)=\partial_l\,\gamma^{-1}_{i,j}(x),\ \ \ \ \ \ \ h^i_{l,j}(x):=\partial_l\,(\gamma^{-1}_{i,j}\sigma_{j,k})(x)A^{-1}_{k,h},\]
we obtain \eqref{smf95} once we can show 
\begin{equation}
\label{smf102}
\lim_{n\to \infty}\mathbb{E}\sup_{t \in\,[0,T]}\left|\alpha_n \int_0^t ([D\gamma^{-1}(x_{n}(r))\,\tilde{u}_n(s)]\tilde{u}_n(s))_i\,ds-\int_0^t g^i_{l,j}(x(s))[\alpha\,N_\alpha(x(s))]_{l,j}\,ds\right|=0.
\end{equation}
and\begin{equation}
\label{smf103}
\begin{array}{l}\ds{\lim_{n\to \infty}\mathbb{E}\sup_{t \in\,[0,T]}\Big|\int_0^t ([D(\gamma^{-1}\sigma)(x_{n}(r))\,\tilde{u}_{n}(r)]\,A^{-1}\tilde{z}_n(r))_i\,dr}\\[14pt]
\ds{\quad \quad \quad \quad \quad \quad\quad \quad \quad \quad \quad \quad\quad \quad \quad \quad \quad \quad-\int_0^t h^i_{l,j}(x(s))[L_\alpha(x(s))]_{l,j}\,ds\Big|=0.}
\end{array}
\end{equation}
Thus, in order to prove \eqref{smf95}, and hence Theorem \ref{teo2}, we need to prove the following lemma.
\begin{Lemma} \label{lemma5.1}Assume that \eqref{smf105-bis} holds for $\alpha \in\,(0,+\infty)$ and that $x_n$ converges to $x$ in $C([0,T];\mathbb{R}^d)$, $\mathbb{P}$-a.s. Then, 
for every  Lipschitz-continuous and bounded function $g:\mathbb{R}^d\to\mathbb{R}$ and every $l, j=1,\ldots,d$ and $k=1,\ldots,n$ we have
\begin{equation}
\label{smf100}
\lim_{n\to \infty}\mathbb{E}\sup_{t \in\,[0,T]}\left| \alpha_n\int_0^t g(x	_n(s))[\tilde{u}_n(s)]_l[\tilde{u}_n(s)]_j\,ds-\int_0^t g(x(s))[\alpha N_\alpha(x(s))]_{l,j}\,ds\right|=0,
\end{equation}
and
\begin{equation}
\label{smf101}
\lim_{n\to \infty}\mathbb{E}\sup_{t \in\,[0,T]}\left|\int_0^t g(x	_n(s))[\tilde{u}_n(s)]_l[\tilde{z}_n(s)]_k\,ds-\int_0^t g(x(s))\,[L_\alpha(x(s))]_{l,k}\,ds\right|=0.
\end{equation}
\end{Lemma}
\begin{proof}
In what follows we only prove \eqref{smf100}, as the proof of \eqref{smf101} can be done by using the very same arguments.

Assume that $\delta_n:=\mu_n \zeta_n$, for some $\zeta_n>0$ such that
\begin{equation}\label{smf110}\lim_{n\to\infty}\zeta_n=\infty,\ \ \ \ \ \lim_{n\to\infty} \zeta_n^2\left(\mu_n\zeta_n+\mu_n+\e_n\right)=0.\end{equation}
We denote by  $\bar{x}_n(t)$ the process
\[\bar{x}_n(t)=x_n(k \delta_n),\ \ \ \ \ t \in\,[k\delta_n,(k+1)\delta_n),\ \ \  k=0,\ldots,[T/\delta_n]-1,\]
and by $\bar{u}_n$ the solution of the equation
\[\left\{\begin{array}
	{l}
	\ds{\frac{d\bar{u}_{n}}{dt}(t)=\frac{1}{\mu_n}\Big(-\gamma(x_{n}(k\delta_n))\,\bar{u}_{n}(t)+\sigma(x_{n}(k\delta_n))\,\tilde{z}_n(t)\Big),\ \ \ \ \ \ t \in\,[k\delta_n,(k+1)\delta_n),}\\[14pt]
	\ds{\bar{u}_n(k\delta_n)=\tilde{u}_n(k\delta_n),}
\end{array}\right.\]
for every $k=0,\ldots,[T/\delta_n]$.

\medskip

{\em Step 1.} 
We have
\begin{equation}
\label{smf107-bis}
\lim_{n\to \infty} \sup_{t \in\,[0,T]}\mathbb{E}\vert x_n(t)-\bar{x}_n(t)\vert^2=0,
\end{equation}
and
\begin{equation}
\label{smf107}
\lim_{n\to \infty} \sup_{t \in\,[0,T]}\mathbb{E}\vert \tilde{u}_n(t)-\bar{u}_n(t)\vert^2=0.	
\end{equation}

Due to \eqref{smf106},
limit \eqref{smf107-bis} holds with the only assumption that $\delta_n\to 0$, As for \eqref{smf107}, if we define $\Delta_n(t):=\tilde{u}_n(t)-\bar{u}_n(t)$, for every $k=0,\ldots,[T/\delta_n]-1$ and $t \in\,[k\delta_n,(k+1)\delta_n)$ we have
\[\begin{array}{l}
\ds{\Delta_n(t)=-\frac{1}{\mu_n}\int_{k\delta_n}^{	t}\gamma(x_{n}(k\delta_n))\Delta_n(s)\,ds-\frac{1}{\mu_n}\int_{k\delta_n}^{t}\left(\gamma(x_{n}(s))-\gamma(x_{n}(k\delta_n))\right)\tilde{u}_n(s)\,ds
}\\[14pt]
\ds{\quad  
+\frac{1}{\mu_n}\int_{k\delta_n}^{	t}\left(\sigma(x_{n}(s))-\sigma(x_{n}(k\delta_n))\right)\tilde{z}_n(s)\,ds-\int_{k\delta_n}^{	t}\left(\sqrt{\e_n}\,S_1(x_{n}(s))+S_2(x_{n}(s))\,\tilde{u}_n(s)\right)\,ds.}\\\end{array}\]
Hence, it holds
\[\begin{array}{l}
\ds{\Delta_n(t)=-\frac{1}{\mu_n}\int_{k\delta_n}^{	t}e^{-\frac{\gamma(x_{n}(k\delta_n))}{\mu_n}(t-s)}\left(\gamma(x_{n}(s))-\gamma(x_{n}(k\delta_n))\right)\tilde{u}_n(s)\,ds}\\[14pt]
\ds{\quad \quad \quad \quad +\frac{1}{\mu_n}\int_{k\delta_n}^{	t}e^{-\frac{\gamma(x_{n}(k\delta_n))}{\mu_n}(t-s)}\left(\sigma(x_{n}(s))-\sigma(x_{n}(k\delta_n))\right)\tilde{z}_n(s)\,ds}\\[14pt]
\ds{\quad \quad \quad \quad \quad \quad -\int_{k\delta_n}^{	t}e^{-\frac{\gamma(x_{n}(k\delta_n))}{\mu_n}(t-s)}\left(\sqrt{\e_n}\,S_1(x_{n}(s))+S_2(x_{n}(s))\,\tilde{u}_n(s)\right)\,ds.}	
\end{array}\]
Due to the Lipschitz continuity of $\gamma$ and  $\sigma$ and estimates \eqref{R(x,u)}, this implies
\[\begin{array}{l}
\ds{\vert\Delta_n(t)\vert\leq  \frac{c}{\mu_n}\int_{k\delta_n}^{	t}e^{-\frac{\gamma_0}{\mu_n}(t-s)}\left|x_{n}(s)-x_{n}(k\delta_n)\right|\left(\vert \sqrt{\e_n}\,u_n(s)\vert+\vert\sqrt{\e_n}\,z_n(s)\vert\right)\,ds}\\[14pt]
\ds{\quad \quad \quad \quad \quad \quad
+c\int_{k\delta_n}^{	t}e^{-\frac{\gamma_0}{\mu_n}(t-s)}\left(1+\vert x_{n}(s)\vert^2+\vert\sqrt{\e_n}\,u_n(s)\vert^2\right)\,ds,}	
\end{array}\]
and in view of \eqref{smf26}, \eqref{smf25},\eqref{smf80}, and \eqref{smf106}, we get
\begin{equation}
\label{smf130}
\mathbb{E}	\vert\Delta_n(t)\vert^2\leq c_T\left(\delta_n+\mu_n+\epsilon_n\right)\left(\frac{\delta_n}{\mu_n}\right)^2.
\end{equation}
In particular, if we assume
$\delta_n>0$ satisfies \eqref{smf110},
we obtain \eqref{smf107}.

\medskip

{\em Step 2.} We have
\begin{equation}\label{smf143}
\lim_{n\to\infty}\mathbb{E}\sup_{t \in\,[0,T]}\left|\alpha_n\int_0^t \left(g(x	_n(s))-g(\bar{x}_n(s))\right)[\tilde{u}_n(s)]_l[\tilde{u}_n(s)]_j\,ds\right|=0.\end{equation}

Due to the Lipschitz-continuity of $g$, we have
\[\begin{array}{l}
\ds{\alpha_n\int_0^t \left|g(x	_n(s))-g(\bar{x}_n(s))\right|\,\vert[\tilde{u}_n(s)]_l\vert\,\vert[\tilde{u}_n(s)]_j\vert\,ds}\\[14pt]
\ds{\quad \quad \quad \leq c\,\int_0^t \left|x	_n(s)-\bar{x}_n(s)\right|\,\vert[\sqrt{\mu_n}\,u_n(s)]_l\vert\,\vert[\sqrt{\mu_n}\,u_n(s)]_j\vert\,ds.}	
\end{array}\]
Hence, thanks to \eqref{smf33} and \eqref{smf107-bis}, we obtain \eqref{smf143}.

\medskip

{\em Step 3.} For every $l,j=1,\ldots,d$ we have
\begin{equation}
\label{smf111}
\lim_{n\to\infty} \sum_{k=0}^{[T/\delta_n]-1}\mathbb{E}\,\Big|\int_{k\delta_n}^{(k+1)\delta_n}	\left( \alpha_n\, g(x_n(k\delta_n))[\bar{u}_n(s)]_l[\bar{u}_n(s)]_j- g(x_n(k\delta_n))[\alpha_n\,N_{\alpha_n}(x(k\delta_n))]_{l,j}\right)\,ds\Big|=0.
\end{equation}

The distribution of the process 
\[(u_{1,n}(s),z_{1,n}(s)):=(\bar{u}_n(k\delta_n+s),\tilde{z}_n(k\delta_n+s)),\ \ \ \ \ s \in\,[0,\delta_n],\]
is the same as the distribution of the process
\[(u_{2,n}(s),z_{2,n}(s)):=(u^{x_n(k\delta_n),\tilde{u}_n(k\delta_n)}(s/\e_n),z^{\tilde{z}_n(k\delta_n)}(s/\e_n)),\ \ \ \ \ s \in\,[0,\delta_n],\]
where $(u^{x_n(k\delta_n),\tilde{u}_n(k\delta_n)},z^{\tilde{z}_n(k\delta_n)})$ is the solution of system \eqref{smf4} for $\alpha=\alpha_n$, with frozen slow component $x_n(k\delta_n)$, initial conditions $\tilde{u}_n(k\delta_n)$ and $\tilde{z}_n(k\delta_n)$ and noise $\tilde{w}_n(t)$, independent of $x_n(k\delta_n)$, $\tilde{u}_n(k\delta_n)$ and $\tilde{z}_n(k\delta_n)$.
Hence
\[\begin{array}{l}
\ds{\mathbb{E}\,\Big|\int_{k\delta_n}^{(k+1)\delta_n}	\left(\alpha_n\,g(x	_n(k\delta_n))[\bar{u}_n(s)]_l[\bar{u}_n(s)]_j- g(x_n(k\delta_n))[\alpha_n\,N_{\alpha_n}(x_n(k\delta_n))]_{l,j}\right)\,ds\Big|}\\[14pt]
\ds{\quad =\mathbb{E}\,\Big|\int_{0}^{\delta_n}	\left(\alpha_n\,  g(x	_n(k\delta_n))[u_{1,n}(s)]_l[u_{1,n}(s)]_j- g(x_n(k\delta_n))[\alpha_n\,N_{\alpha_n\,}(x_n(k\delta_n))]_{l,j}\right)\,ds\Big|}	\\[14pt]
\ds{\quad =\mathbb{E}\,\Big|\int_{0}^{\delta_n}	\left(\alpha_n\, g(x	_n(k\delta_n))[u_{2,n}(s)]_l[u_{2,n}(s)]_j- g(x_n(k\delta_n))[\alpha_n\,N_{\alpha_n}(x_n(k\delta_n))]_{l,j}\right)\,ds\Big|}\\[14pt]
\ds{\quad  =\delta_n\,\alpha_n\mathbb{E}\,\Big|\frac 1{\theta_n}\int_{0}^{\theta_n}	\Big( g(x	_n(k\delta_n))[u^{x_n(k\delta_n),\tilde{u}_n(k\delta_n)}(s)]_l[u^{x_n(k\delta_n),\tilde{u}_n(k\delta_n)}(s)]_j}\\[14pt]
\ds{\quad \quad \quad \quad \quad \quad \quad \quad \quad \quad \quad \quad \quad \quad - g(x_n(k\delta_n))[N_{\alpha_n}(x_n(k\delta_n))]_{l,j}\Big)\,ds\Big|,}
\end{array}\]
where
\[\theta_n:=\frac{\delta_n}{\e_n}=\frac{\delta_n}{\mu_n}\alpha_n= \zeta_n\,\alpha_n.\]
Thanks to \eqref{smf15}, and to \eqref{smf26} and \eqref{smf33},  this implies 
\[\begin{array}{l}
\ds{\mathbb{E}\,\Big|\int_{k\delta_n}^{(k+1)\delta_n}	\left( \alpha_n\,g(x	_n(k\delta_n))[\bar{u}_n(s)]_l[\bar{u}_n(s)]_j- g(x_n(k\delta_n))[\alpha_n\,N_{\alpha_n}(x(k\delta_n))]_{l,j}\right)\,ds\Big|,}\\[14pt]
\ds{\quad \quad \quad \quad \quad \quad \leq c\,\frac{\delta_n\,\alpha_n}{\sqrt{\omega_{\alpha_n}\theta_n}}\left(1+\mathbb{E}\,\vert\tilde{u}_n(k\delta_n)\vert + \mathbb{E}\,\vert \tilde{z}_n(k\delta_n)\vert\right)\leq c\,\frac{\delta_n}{\sqrt{\omega_{\alpha_n}\theta_n}},}	
\end{array}
\]
so that
\[\begin{array}{l}
\ds{\sum_{k=0}^{[T/\delta_n]-1}\mathbb{E}\,\Big|\int_{k\delta_n}^{(k+1)\delta_n}	\left( \alpha_n\, g(x_n(k\delta_n))[\bar{u}_n(s)]_l[\bar{u}_n(s)]_j- g(x_n(k\delta_n))[\alpha_n\,N_{\alpha_n}(x(k\delta_n))]_{l,j}\right)\,ds\Big|}\\[14pt]
\ds{\quad \quad \quad \quad \quad \quad \quad \quad \quad \quad \quad \quad \leq c\,[T/\delta_n]\,\frac{\delta_n}{\sqrt{\omega_{\alpha_n}\theta_n}} \leq \frac{c_T\,}{\sqrt{\omega_{\alpha_n}\theta_n}}.}
	\end{array}\]

Now, we recall that 
\[\omega_{\alpha_n}:=\min\left\{ \frac {\gamma_0} {\alpha_n}, \bar{\lambda}\right\}.\]	
If $\alpha\leq \gamma_0/\bar{\lambda}$, then for every $n$ large enough we have $\omega_{\alpha_n}\geq \bar{\lambda}/2$ and $\theta_n\geq \zeta_n\,\alpha/2$, so that
\[\frac{1}{\omega_{\alpha_n}\theta_n}\leq \frac{4 }{\bar{\lambda}\,\alpha\,\zeta_n}\to 0. \]
On the other hand, if $\alpha>\gamma_0/\bar{\lambda}$, then  for all $n$ large enough we have $\omega_{\alpha_n}=\gamma_0/\alpha_n$, so that 
\[\frac{1}{\omega_{\alpha_n}\theta_n}=\frac{\alpha_n}{\gamma_0\,\alpha_n\,\zeta_n}=\frac{1}{\gamma_0\,\zeta_n}\to 0. \]
All this implies \eqref{smf111}.

\medskip

{\em Step 4.} We have
\begin{equation}
\label{smf123-bis}
\lim_{n\to \infty}\mathbb{E}\sup_{t \in\,[0,T]}\left| \int_0^t \left(g(\bar{x}_n(s))-g(x(s))\right)\,[\alpha_n N_{\alpha_n}(x(s))]_{l,j}\,ds\right|=0.
\end{equation}

Due to \eqref{smf107-bis} and the convergence of $x_n$ to $x$, we have
\[\lim_{n\to\infty}\sup_{t \in\,[0,T]}\left| \int_0^t \left(g(\bar{x}_n(s))-g(x(s))\right)\,[\alpha_n N_{\alpha_n}(x(s))]_{l,j}\,ds\right|=0.\]
Thus, since
\begin{equation}\label{smf140}\sup_{n \in\,\mathbb{N}}\,\sup_{s \in\,[0,T]}\,\left\vert [\alpha_n N_{\alpha_n}(x(s))]_{l,j}\right\vert<\infty,\end{equation}
by the dominated convergence theorem, we obtain \eqref{smf123-bis}.

\medskip

{\em Step 5.} We have
\begin{equation}
\label{smf123}
\lim_{n\to \infty}\mathbb{E}\sup_{t \in\,[0,T]}\left| \int_0^t g(x(s))[\alpha_n N_{\alpha_n}(x(s))]_{l,j}\,ds-\int_0^t g(x(s))[\alpha N_\alpha(x(s))]_{l,j}\,ds\right|=0.
\end{equation}

For every  $l, j =1,\ldots,d$, and $\alpha \in\,(0,+\infty]$, we have
\[\lim_{n\to\infty}\sup_{s \in\,[0,T]}\left| g(x(s))[\alpha_n N_{\alpha_n}(x(s))]_{l,j}-g(x(s))[\alpha N_{\alpha}(x(s))]_{l,j}\right|=0,\ \ \ \ \ \ \mathbb{P}-\text{a.s}.\]
Then, thanks to \eqref{smf140}
the dominated convergence theorem implies \eqref{smf123}.

\medskip

{\em Conclusion.} From the combination of \eqref{smf143}, \eqref{smf111}, \eqref{smf123-bis}, and \eqref{smf123}, we obtain \eqref{smf100}.

\end{proof}


\subsection{Case $\alpha=0$}
When $\alpha_n\to0$, we cannot proceed with the same arguments that we used in the case $\alpha \in\,(0,+\infty)$, because the time intervals $[k \delta_n,(k+1)\delta_n]$ where we can approximate $\tilde{u}_n$ by the localized process $\bar{u}_n$, are too short to allow the validity of the averaging limit for 
\begin{equation}\label{smf134}\int_{k\delta_n}^{(k+1)\delta_n} D(\gamma^{-1}\sigma)(x_{n}(r))\,\tilde{u}_{n}(r)]\,A^{-1}\tilde{z}_n(r)\,dr.\end{equation}
For this reason, we need to follow  an alternative route.

If we denote
\[v_n(t)=v_{\epsilon_n}(t),\ \ \ \ \ \ z_n(t)=z_{\epsilon_n}(t),\]
since $\alpha_n\to 0$, by proceeding as in Subsection \ref{ss3.2} we have
\[\begin{array}{l}\ds{x_n(t)=x_0+\int_0^t(\gamma^{-1}b)(x_n(s))\,ds+\int_0^t (\gamma^{-1}\sigma)(x_n(s))\,A^{-1}B\,dw(s)}\\[14pt]
\ds{\quad \quad \quad \quad +\epsilon_n\, \int_0^t \frac{d}{ds}((\gamma^{-1}\sigma)(x_n(s))) A^{-1}z_n(s)\,ds+r_{n,1}(s),}\end{array}\]
for some process $r_{n,1}(t)$ such that 
\begin{equation}\label{smf201-bis}\lim_{n\to \infty}\,\mathbb{E}\sup_{t \in\,[0,T]}\vert r_{n,1}(t)\vert^2=0.\end{equation}
Recalling that 
\[v_n(t)=\left[(\gamma^{-1}b)(x_n(t))+(\gamma^{-1}\sigma)(x_n(t))z_n(t)\right]-\mu_n\,\gamma^{-1}(x_n(t))\frac{dv_n}{dt}(t),\]
 we have 
\[\begin{array}{l}
\ds{\epsilon_n\, d((\gamma^{-1}\sigma)(x_n))(t) A^{-1}z_n(t)=\epsilon_n\, (D(\gamma^{-1}\sigma)(x_n(t))v_n(t)) A^{-1}z_n(t)\,dt}\\[14pt]
\ds{\quad \quad =\epsilon_n\, \left(D(\gamma^{-1}\sigma)(x_n(t))\left[(\gamma^{-1}b)(x_n(t))+\,(\gamma^{-1}\sigma)(x_n(t))\,z_n(t)\right]\right)A^{-1}z_n(t)\,dt}\\[14pt]
\ds{\quad \quad \quad \quad -\mu_n\,\epsilon_n \left(D(\gamma^{-1}\sigma)(x_n(t))\gamma^{-1}(x_n(t)\,dv_n(t)\right)A^{-1}z_n(t)\,dt}\\[14pt]
\ds{=\epsilon_n\, \left(D(\gamma^{-1}\sigma)(\gamma^{-1}\sigma)(x_n(t))\,z_n(t)\right)A^{-1}z_n(t)\,dt}\\[14pt]
\ds{\quad \quad \quad -\mu_n\,\epsilon_n \left(D(\gamma^{-1}\sigma)(x_n(t))\gamma^{-1}(x_n(t)\,dv_n(t)\right)A^{-1}z_n(t)\,dt+r_{n,2}(t)\,dt,}	
\end{array}\]
with
\begin{equation}\label{smf201-tris}\lim_{n\to \infty}\,\mathbb{E}\sup_{t \in\,[0,T]}\vert r_{n,2}(t)\vert^2=0.\end{equation}

Moreover,
\[\begin{array}{l}
\ds{\mu_n\,\epsilon_n\left(D(\gamma^{-1}\sigma)(x_n(t))\gamma^{-1}(x_n(t))\,dv_n(t)\right)A^{-1}z_n(t)}\\[14pt]
\ds{\quad \quad =\mu_n\,\epsilon_nd\left(\left(D(\gamma^{-1}\sigma)(x_n(t))\gamma^{-1}(x_n(t))\,v_n(t)\right)A^{-1}z_n(t)\right)}\\[14pt]
\ds{\quad \quad \quad \quad -\mu_n\,\epsilon_n\Big(\left(D^2(\gamma^{-1}\sigma)(x_n(t))[v_n(t),\gamma^{-1}(x_n(t))\,v_n(t)]\right)A^{-1}z_n(t)\Big)}	\\[14pt]
\ds{\quad \quad \quad \quad \quad \quad -\mu_n\,\epsilon_n\Big(\left(D(\gamma^{-1}\sigma)(x_n(t))[D(\gamma^{-1}\sigma)(x_n(t))v_n(t)]v_n(t)\right)A^{-1}z_n(t)\Big)}\\[14pt]
\ds{\quad \quad +\frac{\mu_n}{\sqrt{\e_n}}\left(D(\gamma^{-1}\sigma)(x_n(t))\gamma^{-1}(x_n(t))\,v_n(t)\right)\tilde{z}_n(t)\,dt}\\[14pt]
\ds{\quad \quad \quad \quad \quad \quad\quad \quad-\mu_n\left(D(\gamma^{-1}\sigma)(x_n(t))\gamma^{-1}(x_n(t))\,v_n(t)\right)dw(t).}\end{array}\]
By checking at each term on the right hand side above, it is not difficult to get
\[\lim_{n\to\infty}\mathbb{E}\sup_{t \in\,[0,T]}\Big|\mu_n\,\epsilon_n\int_0^t\left(D(\gamma^{-1}\sigma)(x_n(s))\gamma^{-1}(x_n(s))\,\frac{dv_n}{ds}(s)\right)A^{-1}z_n(s)\,ds\Big|^2=0,\]
and this, together with  \eqref{smf201-bis} and \eqref{smf201-tris} implies
\[\begin{array}{l}\ds{x_n(t)=x_0+\int_0^t(\gamma^{-1}b)(x_n(s))\,ds+\int_0^t (\gamma^{-1}\sigma)(x_n(s))\,A^{-1}B\,dw(s)}\\[14pt]
\ds{\quad \quad \quad \quad + \left(D(\gamma^{-1}\sigma)(\gamma^{-1}\sigma)(x_n(t))\,\tilde{z}_n(t)\right)A^{-1}\tilde{z}_n(t)\,dt+r_{n}(t),}\end{array}\]
with 
\begin{equation}\label{smf201-quater}\lim_{n\to \infty}\,\mathbb{E}\sup_{t \in\,[0,T]}\vert r_{n}(t)\vert^2=0.\end{equation}

In particular, in order to complete the proof of  Theorem \ref{teo2}, it is enough to show for every $j,h =1,\ldots,n$
\begin{equation}\label{smf135}
\lim_{n\to \infty}\mathbb{E}\sup_{t \in\,[0,T]}\left|\int_0^t g(x	_n(s))[\tilde{z}_n(s)]_j[\tilde{z}_n(s)]_h\,ds-\int_0^t g(x(s))\,[L_0(x(s))]_{j,h}\,ds\right|=0,\end{equation}
where
\[g=\partial_l(\gamma^{-1}\sigma)_{i,k}(\gamma^{-1}\sigma)_{l,j}A^{-1}_{k,h}.\]
Now, the proof of \eqref{smf135} follows from using the same arguments used in the proof of Lemma \ref{lemma5.1}, after localizing $x_n(t)$ in small intervals $[k\delta_n, (k+1)\delta_n]$, $k=0,\ldots,[T/\delta_n]-1$, for some $\delta_n\to 0$ such that $\delta_n/\e_n\to\infty.$ And this is clearly possible, due to \eqref{smf106}.

\subsection{Case $\alpha=\infty$}
We recall that  
\begin{equation}
 \label{system1-tris} 
 \left\{\begin{array}{l}
\ds{\frac{dx_n}{dt}(t)=v_n(t),\ \ \ \ x_n(0)=x_0,}\\[10pt]
\ds{\mu_n \frac{dv_n}{dt}(t)=-\gamma(x_n(t))\,v_n(t)+b(x_n(t))+\sigma(x_n(t))\,z_{n}(t),\ \ \ \ v_n(0)=v_0,}\end{array}\right.
	\end{equation}
	and
	\[\e_n\,dz_{n}(t)=-Az_{n}(t)\,dt+B\,dw(t),\ \ \ \ z_{n}(0)=0.\]
Now, we introduce the system
\begin{equation}
 \label{system1-quater} 
 \left\{\begin{array}{l}
\ds{\frac{d\bar{x}_n}{dt}(t)=\bar{v}_n(t),\ \ \ \ \bar{x}_n(0)=x_0,}\\[10pt]
\ds{\mu_n \frac{d\bar{v}_n}{dt}(t)=-\gamma(\bar{x}_n(t))\,\bar{v}_n(t)+b(\bar{x}_n(t))+\sigma(\bar{x}_n(t))A^{-1}B\,dw(t),\ \ \ \ v_n(0)=v_0.}\end{array}\right.
	\end{equation}
If we can prove that
\begin{equation}
\label{smf200}
\lim_{n\to\infty} \mathbb{E}\sup_{t \in\,[0,T]}\vert x_n(t)-\bar{x}_n(t)\vert^2=0,	
\end{equation}
	we have that that the limiting behavior of $x_n$ in $L^2(\Omega;C([0,T];\mathbb{R}^d))$ is the same as the limiting behavior of $\bar{x}_n$. Therefore, since we know from \cite{hmdvw} that  $\bar{x}_n$ converges to the solution of equation \eqref{smf120}, with $\alpha=\infty$, we conclude our proof.

\subsubsection{Proof of  \eqref{smf200}}

We have
\[\begin{array}{l}
\ds{d x_n(t)=-\mu_n\,\gamma^{-1}(x_n(t))dv_n(t)+(\gamma^{-1}b)(x_n(t))\,dt+(\gamma^{-1}\sigma)(x_n(t))z_{n}(t)\,dt}\\[14pt]
\ds{\quad \quad = \mu_n\left[D\gamma^{-1}(x_n(t))v_n(t)\right]v_n(t)\,dt+(\gamma^{-1}b)(x_n(t))\,dt+(\gamma^{-1}\sigma)(x_n(t))A^{-1}B\,dw(t)}\\[10pt]
\ds{-d\left(\mu_n\,\gamma^{-1}(x_n)\,v_n+\epsilon_n\,(\gamma^{-1}\sigma)(x_n)A^{-1}B\,z_{n}\right)(t)+\epsilon_n\,[D(\gamma^{-1}\sigma)(x_n(t))v_n(t)]A^{-1}B\,z_{n}(t)dt,}	
\end{array}\]
and if we rearrange all terms
we get
\begin{equation}\label{smf202}\begin{array}{l}
\ds{d \left(x_n+\mu_n\,\gamma^{-1}(x_n)\,v_n+\epsilon_n\,(\gamma^{-1}\sigma)(x_n)A^{-1}B\,z_{n}\right)(t)}\\[14pt]
\ds{\quad \quad = \mu_n\left[D\gamma^{-1}(x_n(t))v_n(t)\right]v_n(t)\,dt+(\gamma^{-1}b)(x_n(t))\,dt+(\gamma^{-1}\sigma)(x_n(t))A^{-1}B\,dw(t)}\\[14pt]
\ds{\quad \quad \quad \quad \quad \quad +\epsilon_n\,[D(\gamma^{-1}\sigma)(x_n(t))v_n(t)]A^{-1}B\,z_{n}(t)dt.}	
\end{array}\end{equation}
Moreover, by proceeding with  similar arguments, we can  check that
\begin{equation}\label{smf2011}d\bar{x}_n(t)=\mu_n\left[D\gamma^{-1}(\bar{x}_n(t))\,\bar{v}_n(t)\right]\bar{v}_n(t)\,dt+(\gamma^{-1}b)(\bar{x}_n(t))\,dt+(\gamma^{-1}\sigma)(\bar{x}_n(t))A^{-1}B\,dw(t).\end{equation}

Next,  since
\[\begin{array}{l}
\ds{\sigma(x_n(t))\,z_{n}(t)\,dt=-\e_n d(\sigma(x_n)A^{-1}z_{n})(t)}\\[14pt]
\ds{\quad\quad \quad \quad \quad \quad \quad \quad \quad \quad +\e_n [D\sigma(x_n(t))v_n(t)]A^{-1}z_{n}(t)\,dt+\sigma(x_n(t))A^{-1}B\,dw(t),}	
\end{array}\]
we have
\[\begin{array}{l}\ds{\sqrt{\mu_n} d\left(\sqrt{\mu_n}\,v_n+\sqrt{\frac{\e_n}{\mu_n}} \sigma(x_n)A^{-1}\tilde{z}_{n}\right)(t)}\\[16pt]
\ds{\quad =-\frac{\gamma(x_n(t))}{\sqrt{\mu_n}}\,\left(\sqrt{\mu_n}\,v_n(t)+\sqrt{\frac{\e_n}{\mu_n}} \sigma(x_n(t))A^{-1}\tilde{z}_{n}(t)\right)\,dt+b(x_n(t))\,dt+\sigma(x_n(t))A^{-1}B\,dw(t)}\\[14pt]
\ds{\quad \quad \quad \quad \quad \quad \quad \quad \quad +\e_n [D\sigma(x_n(t))v_n(t)]A^{-1}z_{n}(t)\,dt+\frac{\e_n}{\mu_n} \gamma(x_n(t))\sigma(x_n(t))A^{-1}z_{n}(t)\,dt.}
\end{array}
\]
Thus, if we define
\[\rho_n(t):=\sqrt{\mu_n}\,(v_n(t)-\bar{v}_n(t))+\sqrt{\frac{\e_n}{\mu_n}} \sigma(x_n(t))A^{-1}\tilde{z}_{n}(t),\]
it is not difficult to check that  
\[\begin{array}{l}
\ds{d\rho_n(t)=-\frac{\gamma(x_n(t))}{\mu_n}\,\rho_n(t)\,dt-\frac 1{\mu_n}\left(\gamma(x_n(t))-\gamma(\bar{x}_n(t))	\right)\sqrt{\mu_n}\,\bar{v}_n(t)\,dt}\\[14pt]
\ds{\quad \quad +\frac 1{\sqrt{\mu_n}}\,\left(b(x_n(t))-b(\bar{x}_n(t))\right)\,dt+\frac 1{\sqrt{\mu_n}}\,\left(\sigma(x_n(t))-\sigma(\bar{x}_n(t))\right)A^{-1}B\,dw(t)+\frac{1}{\mu_n}\,\theta_n(t)\,dt,}
\end{array}\]
where
\[\theta_n(t):=\sqrt{\frac{\e_n }{\mu_n}}\,\left(\,[D\sigma(x_n(t))\,\mu_n\,v_n(t)]A^{-1}\tilde{z}_{n}(t)+\gamma(x_n(t))\sigma(x_n(t))A^{-1}\tilde{z}_{n}(t)\right).\]
As a consequence of It\^o's formula, we get
\[\begin{array}{l}
\ds{\frac 12\, d\,\vert \rho_n(t)\vert^2=-\frac{1}{\mu_n}\,\langle \gamma(x_n(t))\rho_n(t),\rho_n(t)\rangle\,dt-\frac 1{\mu_n}\langle\left(\gamma(x_n(t))-\gamma(\bar{x}_n(t))	\right)\sqrt{\mu_n}\,\bar{v}_n(t),\rho_n(t)\rangle\,dt}\\[14pt]
\ds{\quad +\frac 1{\sqrt{\mu_n}}\,\langle b(x_n(t))-b(\bar{x}_n(t)),\rho_n(t)\rangle\,dt+\frac 1{\sqrt{\mu_n}}\,\langle \left(\sigma(x_n(t))-\sigma(\bar{x}_n(t))\right)A^{-1}B\,dw(t),\rho_n(t)\rangle}	\\[14pt]
\ds{\quad \quad +\frac 1{2\,\mu_n}\Vert \left(\sigma(x_n(t))-\sigma(\bar{x}_n(t))\right)A^{-1}B\Vert_2^2\,dt+\frac 1{\mu_n}\langle \theta_n(t),\rho_n(t)\rangle\,dt,}
\end{array}\]
and then, to to the strict positivity of $\gamma$ and the Lipschitz continuity of $\gamma$, $b$ and $\sigma$, we get
\[\begin{array}{l}
\ds{\frac 12\, d\,\vert \rho_n(t)\vert^2\leq -\frac{\gamma_0}{2\mu_n}\,\vert \rho_n(t)\vert^2\,dt+\frac c{\mu_n}\vert\gamma(x_n(t))-\gamma(\bar{x}_n(t))	\vert^2\,\vert\sqrt{\mu_n}\,\bar{v}_n(t)\vert^2\,dt}	\\[14pt]
\ds{\quad \quad +c\,\vert b(x_n(t))-b(\bar{x}_n(t)\vert^2\,dt+\frac c{\mu_n}\vert \sigma(x_n(t))-\sigma(\bar{x}_n(t))\vert^2\,dt }\\[14pt]
\ds{\quad \quad \quad \quad +\frac 1{\sqrt{\mu_n}}\,\langle \left(\sigma(x_n(t))-\sigma(\bar{x}_n(t))\right)A^{-1}B\,dw(t),\rho_n(t)\rangle+\frac 1{\mu_n}\,\vert \theta_n(t)\vert^2\,dt}\\[14pt]
\ds{\leq -\frac{\gamma_0}{2\mu_n}\,\vert \rho_n(t)\vert^2\,dt+\frac c{\mu_n}\vert x_n(t)-\bar{x}_n(t)	\vert^2\,\left(1+\vert\sqrt{\mu_n}\,\bar{v}_n(t)\vert^2\right)\,dt}\\[14pt]
\ds{\quad \quad \quad \quad +\frac 1{\sqrt{\mu_n}}\,\langle \left(\sigma(x_n(t))-\sigma(\bar{x}_n(t))\right)A^{-1}B\,dw(t),\rho_n(t)\rangle+\frac 1{\mu_n}\,\vert \theta_n(t)\vert^2\,dt.}
\end{array}\]
This implies for every $R>0$ there exists some $c_R>0$ such that
\[\begin{array}{l}
\ds{\mathbb{E}\vert \rho_n(t)\vert^2\leq \frac c{\mu_n}\int_0^t e^{-\frac{\gamma_0}{\mu_n}(t-s)}\mathbb{E}\,\left(\vert x_n(s)-\bar{x}_n(s)	\vert^2\,\left(1+\vert\sqrt{\mu_n}\,\bar{v}_n(s)\vert^2\right)\right)\,ds}\\[14pt]
\ds{\quad \quad \quad \quad +\frac c{\mu_n}\int_0^t e^{-\frac{\gamma_0}{\mu_n}(t-s)}\mathbb{E}\,\,\vert \theta_n(s)\vert^2\,ds\leq \frac {c_R}{\mu_n}\int_0^t e^{-\frac{\gamma_0}{\mu_n}(t-s)}\mathbb{E}\,\vert x_n(s)-\bar{x}_n(s)	\vert^2\,ds}	\\[14pt]
\ds{\quad \quad \quad \quad \quad +c\,\sup_{s \in\,[0,T]}\mathbb{E}\,\Big(\left(\vert x_n(s)\vert^2+\vert \bar{x}_n(s)	\vert^2\right)\,\left(1+\vert\sqrt{\mu_n}\,\bar{v}_n(s)\vert^2\right)\mathbb{I}_{D_{n,R}}\Big)+c\,\sup_{t \in\,[0,T]}\mathbb{E}\,\,\vert \theta_n(t)\vert^2,}
\end{array}\]
where
\[D_{n,R}:=\left\{\sup_{t \in\,[0,T]}\vert\sqrt{\mu_n}\,\bar{v}_n(t)\vert^2\geq R\right\}.\]

The same bounds that we have been proven for $x_n$ and $v_n$ hold for $\bar{x}_n$ and $\bar{v}_n$, as well. Then, thanks to \eqref{smf33} and \eqref{smf80}, for every $\eta>0$ there exists some $R_\eta>0$ such that 
\[c\,\sup_{s \in\,[0,T]}\mathbb{E}\,\Big(\left(\vert x_n(s)\vert^2+\vert \bar{x}_n(s)	\vert^2\right)\,\left(1+\vert\sqrt{\mu_n}\,\bar{v}_n(s)\vert^2\right)\mathbb{I}_{D_{n,R_\eta}}\Big)<\eta,\ \ \ \ \ n \in\,\mathbb{N}.\]
Moreover, since $\e_n/\mu_n\to 0$, thanks to \eqref{smf26} and \eqref{smf33}, there exists some $n_\eta \in\,\mathbb{N}$ such that
\[\sup_{t \in\,[0,T]}\mathbb{E}\,\vert \theta_n(t)\vert^2\leq \eta,\ \ \ \ \ n\geq n_\eta,\]
and
\[2\,\sqrt{\frac{\e_n}{\mu_n}} \sup_{t \in\,[0,T]}\mathbb{E}\,\vert \sigma(x_n(t))\tilde{z}_{n}(t)\vert^2\leq \eta,\ \ \ \ \ n\geq n_\eta.\]
Thus, we can conclude that
\begin{equation}\label{smf193}\begin{array}{l}
\ds{\mathbb{E}\vert \sqrt{\mu_n}\left(v_n(t)-\bar{v}_n(t)\right)\vert^2\leq 2\,\mathbb{E}\vert \rho_n(t)\vert^2+2\,\sqrt{\frac{\e_n}{\mu_n}} \sup_{t \in\,[0,T]}\mathbb{E}\,\vert \sigma(x_n(t))\tilde{z}_{n}(t)\vert^2}\\[14pt]
\ds{\quad \quad \quad \quad \leq \frac {c_{R_\eta}}{\mu_n}\int_0^t e^{-\frac{\gamma_0}{\mu_n}(t-s)}\mathbb{E}\,\vert x_n(s)-\bar{x}_n(s)	\vert^2\,ds+3\,\eta,\ \ \ \ \ n\geq n_\eta.}\end{array}\end{equation}

Now, if we set
\[\lambda_n(t):=(x_n(t)-\bar{x}_n(t))+\mu_n\,\gamma^{-1}(x_n(t))\,v_n(t)+\epsilon_n\,(\gamma^{-1}\sigma)(x_n(t))A^{-1}B\,z_{n}(t),\]
in view of \eqref{smf202}, we have
\[\begin{array}{l}
\ds{d\lambda_n(t)=\left((\gamma^{-1}b)({x}_n(t))-(\gamma^{-1}b)(\bar{x}_n(t))\right)\,dt+\left((\gamma^{-1}\sigma)(x_n(t))-(\gamma^{-1}\sigma)(\bar{x}_n(t))\right)A^{-1}B\,dw(t)}\\[14pt]
\ds{+ \left(\left[D\gamma^{-1}(x_n(t))\,\sqrt{\mu_n}\,v_n(t)\right]\sqrt{\mu_n}\,v_n(t)-\left[D\gamma^{-1}(\bar{x}_n(t))\,\sqrt{\mu_n}\,\bar{v}_n(t)\right]\sqrt{\mu_n}\,\bar{v}_n(t)\right)\,dt+\theta_n(t)\,dt,}\end{array}\]
where
\[\begin{array}{l}
\ds{\theta_n(t):=\epsilon_n[D(\gamma^{-1}\sigma)(x_n(t))v_n(t)]A^{-1}B\,z_{n}(t).}	
\end{array}\]
Due to the boundedness and Lipschitz-continuity of $D\gamma^{-1}$, for every $R>0$ we have
\[\begin{array}{l}
\ds{\vert\left[D\gamma^{-1}(x)\,v\right]v-\left[D\gamma^{-1}(\bar{x})\,\bar{v}\right]\bar{v}\vert \leq c\,\vert x-\bar{x}\vert \,\vert v\vert^2+c\,\vert v-\bar{v}\vert\left(\vert v\vert+\vert \bar{v}\vert\right)}	\\[14pt]
\ds{\quad \quad \quad \quad \quad \leq c_R\left(\vert x-\bar{x}\vert + \vert v-\bar{v}\vert\right)+c\left(\vert x\vert+\vert \bar{x}\vert+1\right)\left(1+\vert v\vert^2+\vert \bar{v}\vert^2\right)\mathbb{I}_{\left\{\vert v\vert \geq R\right\}\cup\left\{\vert \bar{v}\vert \geq R\right\}}.}
\end{array}\]
Hence
\[\begin{array}{l}
\ds{\mathbb{E}\sup_{s \in\,[0,t]}\vert\lambda_n(s)\vert^2 \leq c_{T,R}\int_0^t \mathbb{E}\sup_{r \in\,[0,s]} \vert x_n(r)-\bar{x}_n(r)\vert^2\,ds+c_{T,R}\int_0^t \mathbb{E}\,\vert \sqrt{\mu_n}\,v_n(s)- \sqrt{\mu_n}\,\bar{v}_n(s)\vert^2\,ds}\\[16pt]
\ds{+c_T\sup_{t \in\,[0,T]}\mathbb{E}\Big[\left(1+\vert x_n(t)\vert+\vert \bar{x}_n(t)\vert	\right)\left(1+\vert\sqrt{\mu_n}\,v_n(t)\vert^2+\vert\sqrt{\mu_n}\,\bar{v}_n(t)\vert^2\right)\mathbb{I}_{D_{n,R}(t)}\Big],}
\end{array}\]
where
\[D_{n,R}(t)=\left\{\vert v_n(t)\vert \geq R\right\}\cup\left\{\vert \bar{v}_n(t)\vert \geq R\right\}.\]
As the uniform bounds \eqref{smf33} and \eqref{smf80} holds also for $\bar{x}_n$ and $\bar{v}_n$, for every $\eta>0$ we can find $R_\eta>0$ such that
\[c_T\sup_{t \in\,[0,T]}\mathbb{E}\Big[\left(1+\vert x_n(t)\vert+\vert \bar{x}_n(t)\vert	\right)\left(1+\vert\sqrt{\mu_n}\,v_n(t)\vert^2+\vert\sqrt{\mu_n}\,\bar{v}_n(t)\vert^2\right)\mathbb{I}_{D_{n,R_\eta}(t)}\Big]<\eta,\ \ \ \ n \in\,\mathbb{N}.\]
Therefore, since 
\[\begin{array}{l}
\ds{\lim_{n\to\infty}\mathbb{E}\sup_{t \in\,[0,T]}\vert \mu_n\,\gamma^{-1}(x_n(t))\,v_n(t)+\epsilon_n\,(\gamma^{-1}\sigma)(x_n(t))A^{-1}B\,z_{n}(t)\vert^2=0,}
\end{array}
\]
we can fix $n_\eta \in\,\mathbb{N}$ such that
\[\begin{array}{l}
\ds{\mathbb{E}\sup_{s \in\,[0,t]}\vert x_n(s)-\bar{x}_n(s)\vert^2 \leq c_{T,R_\eta}\int_0^t \mathbb{E}\sup_{r \in\,[0,s]} \vert x_n(r)-\bar{x}_n(r)\vert^2\,ds}\\[14pt]
\ds{\quad \quad \quad \quad +c_{T,R_\eta}\int_0^t \mathbb{E}\,\vert \sqrt{\mu_n}\,v_n(s)- \sqrt{\mu_n}\,\bar{v}_n(s)\vert^2\,ds+2\eta,\ \ \ \ \ n\geq n_\eta,}\end{array}\]
and Gronwall's lemma gives
\begin{equation}
\label{smf195}
\mathbb{E}\sup_{s \in\,[0,t]}\vert x_n(s)-\bar{x}_n(s)\vert^2 \leq c_T\int_0^t \mathbb{E}\,\vert \sqrt{\mu_n}\,v_n(s)- \sqrt{\mu_n}\,\bar{v}_n(s)\vert^2\,ds+c_T\,\eta,\ \ \ \ \ n\geq n_\eta.
\end{equation}

Finally, if we replace \eqref{smf195} into \eqref{smf193}, we obtain 
\[\mathbb{E}\vert \sqrt{\mu_n}\left(v_n(t)-\bar{v}_n(t)\right)\vert^2\leq  c_T\int_0^t \mathbb{E}\vert \sqrt{\mu_n}\left(v_n(s)-\bar{v}_n(s)\right)\vert^2\,ds+c_T\,\eta,\ \ \ \ \ \ n\geq n_\eta,\]
which gives
\[\limsup_{n\to\infty}\sup_{t \in\,[0,T]}\mathbb{E}\vert \sqrt{\mu_n}\left(v_n(t)-\bar{v}_n(t)\right)\vert^2\leq c_T\eta.\]
By plugging this into \eqref{smf195}, due to the arbitrariness of $\eta>0$, we get \eqref{smf200}.

\section{Examples of drifts in turbulence fluids\label{subsect examples drift}%
}

Particles in turbulent flows are transported, displaced, deposited and
aggregated by complex interactions with the turbulent flow. The inertia of
particles, captured by the second-order nature of the equations of motion and
the positive mass constant $\mu$, is responsible for a number of effects. When
we take the limit as $\mu\rightarrow0$, \textit{certain forces show up as
drifts}. We discuss two of them here:\ a drift corresponding to the
centrifugal force, and the so-called turbo-phoretic drift. 


\subsubsection{A model of turbulence}

In this subsection, we introduce a general model of particle motion in a
turbulent fluid. The model takes the form%
\[\left\{\begin{array}{l}
\ds{\frac{dx}{dt}\left(  t\right) =v\left(  t\right)} \\[10pt]
\ds{\mu\frac{dv}{dt}\left(  t\right)     =-c_{0}k_{T}\left(  x\left(  t\right)
\right)  \Big(  v\left(  t\right)  -\overline{u}\left(  x\left(  t\right)
\right)  -\sum_{k\in K}\xi_{k}\left(  x\left(  t\right)  \right)  z^{k}\left(
t\right)  \Big)}  \\[14pt]
\ds{\epsilon dz^{k}\left(  t\right)     =-z^{k}\left(  t\right)  dt+dw^{k}\left(
t\right) } \\[10pt]
\ds{x\left(  0\right)     =x_{0},\qquad v\left(  0\right)  =v_{0},\qquad
z^{k}\left(  0\right)  =0.}
\end{array}\right.\]
Here $x\left(  t\right)  ,v\left(  t\right)  ,x_{0},v_{0}\in\mathbb{R}^{d}$,
$\mu,c_{0},\epsilon>0$, $K$ is a finite index set, $w^{k}\left(  t\right)  $
are independent real-valued Brownian motions (notice  that $z^{k}\left(
t\right)  $ are one-dimensional processes), $k_{T}\left(  x\right)  $ is a
bounded function on $\mathbb{R}^{d}$, $\overline{u}\left(  x\right)  $ and all
$\xi_{k}\left(  x\right)  $, $k\in K$, are bounded vector fields on
$\mathbb{R}^{d}$, $k_{T}$ and $\overline{u}$ are continuously differentiable functions 
with bounded derivatives, $\xi_{k}$ are twice continuously differentiable vector fields  with
bounded derivatives. We assume that there exists $k_{T}^{0}>0$ such that
\[
k_{T}\left(  x\right)  \geq k_{T}^{0}, \ \ \ \ \ \ \ x \in\,\mathbb{R}^d.%
\]
Hence, this system is a particular case of the
general one studied in Theorem \ref{teo2}. Moreover, we assume
that $\xi_{k}$ are divergence free%
\[
\operatorname{div}\xi_{k}=0.
\]
The purpose of this section is to provide a motivation for this model.

First, let us explain the meaning of the terms:\ $x\left(  t\right)  ,v\left(
t\right)  $ are particle position and velocity, $\mu$ is the particle mass,
$k_{T}\left(  x\right)  $ has the meaning of local average turbulent kinetic
energy, $\overline{u}\left(  x\right)  $ has the meaning of local average
fluid velocity, and finally the random vector field $\sum_{k\in K}\xi
_{k}\left(  x\right)  z^{k}\left(  t\right)  $ is our model of the turbulent
fluctuations of the fluid.

The choice of the structure
\begin{equation}
-c_{0}k_{T}\left(  x\left(  t\right)  \right)  \left(  v\left(  t\right)
-\overline{u}\left(  x\left(  t\right)  \right)  -\sum_{k\in K}\xi_{k}\left(
x\left(  t\right)  \right)  z^{k}\left(  t\right)  \right)
\label{our Stokes law}%
\end{equation}
for the force acting on the particle is the \textit{Stokes law} for the
interaction between particle and fluid, which has the form%
\begin{equation}
-\lambda\left(  v\left(  t\right)  -u\left(  x\left(  t\right)  ,t\right)
\right)  \label{classical Stokes law}%
\end{equation}
for a certain constant $\lambda>0$, where $u\left(  x,t\right)  $ is the fluid
velocity. Therefore, we have to explain why we consider a space-dependent
function $c_{0}k_{T}\left(  x\left(  t\right)  \right)  \geq c_{0}k_{T}^{0}>0$
in place of a constant $\lambda>0$ and why the fluid velocity $u\left(
x,t\right)  $ has been replaced by 
\[\overline{u}\left(  x\right)  +\sum_{k\in
K}\xi_{k}\left(  x\right)  z^{k}\left(  t\right).\]

Assume that $u\left(  x,t\right)  $ has the form%
\[
u\left(  x,t\right)  =\overline{u}\left(  x\right)  +u^{\prime}\left(
x,t\right)  +u^{\prime\prime}\left(  x,t\right)
\]
where $\overline{u}\left(  x\right)  $ is slowly varying (the large average
scale, much larger than the particle size), $u^{\prime}\left(  x,t\right)  $
varies fast, but at the same scale of the particle (call it the
\textit{intermediate scales}), and $u^{\prime\prime}\left(  x,t\right)  $
collects variations at a much smaller scale with respect to the particle (the
\textit{smallest turbulent scales}). We then introduce two different models of
$u^{\prime}\left(  x,t\right)  $ and $u^{\prime\prime}\left(  x,t\right)  $,
appropriate for the two different scales.

We model the variations at the particle-scale by the stochastic process%
\[
u^{\prime}\left(  x,t\right)  =\sum_{k\in K}\xi_{k}\left(  x\right)
z^{k}\left(  t\right)  .
\]
This choice is reasonable, following a large literature on synthetic models of
turbulent fluids (among many others, see \cite{Bec}, \cite{Kraichnan},
\cite{Majda}).

As for $u^{\prime\prime}\left(  x,t\right)  $, following Boussinesq intuition
\cite{Bousinnesq}, we adopt the approximation of Large Eddy theory
\cite{Berselli}: very small turbulence produces effects similar to an
additional viscosity, the turbulent viscosity $\nu_{T}\left(  x\right)
$, which is space dependent, in general. If we want to "close" the system, we
may follow the prescriptions of Smagorinsky theory and choose the turbulent
viscosity proportional to $\left\vert D\overline{u}\left(  x\right)
\right\vert $ or $\left\vert \operatorname{curl}\overline{u}\left(  x\right)
\right\vert $, but depending on specific applications, we may have other
information on $\nu_{T}\left(  x\right)  $. For instance, in pipe flows, the
function $\nu_{T}\left(  x\right)  $ depends on the distance from the center,
with good experimental knowledge of the profile. Recall moreover that the
parameter $\lambda$ in Stokes law depends on the viscosity. Up to other
constants, we have $\lambda\sim\nu$, where $\nu$ is the viscosity. Under the
Large Eddy theory modeling based on a space-dependent turbulent viscosity, we
have $\lambda\left(  x\right)  \sim\nu_{T}\left(  x\right)  $. This is the
origin, in our model, of a space-dependent damping coefficient in Stokes law.

Moreover, to be even closer to the statistical properties of turbulence, we may
use the approximation $\nu_{T}\left(  x\right)  \sim k_{T}\left(  x\right)  $
(up to factors, summarized in the constant $c_{0}$), where $k_{T}\left(
x\right)  $ is the \textit{turbulent kinetic energy}, namely the mean square
of the velocity of the smallest turbulent eddies. We therefore replace the
constant $\lambda$ in the classical Stokes law (\ref{classical Stokes law}) by
the space-dependent coefficient $c_{0}k_{T}\left(  x\left(  t\right)  \right)
$ in our model (\ref{our Stokes law}).

Simultaneously, we neglect the term $u^{\prime\prime}\left(  x,t\right)  $ in
the formula for $u\left(  x,t\right)  $ since its effect is already taken into
account by the factor $c_{0}k_{T}\left(  x\left(  t\right)  \right)  $.
Therefore, the fluid model becomes
\[
\overline{u}\left(  x\right)  +\sum_{k\in K}\xi_{k}\left(  x\right)
z^{k}\left(  t\right)
\]
that appears in our version (\ref{our Stokes law}) of the Stokes law.

\subsubsection{The general drift induced by the turbulent flow model}

As already remarked, the particle-in-turbulent-fluid model above fulfills all
the assumptions of the abstract theory, hence we shall now apply the
convergence Theorem 2.2. Assume that the mass $\mu=\mu\left(  \epsilon\right)
$ depends on the parameter $\epsilon$ and call $x_{\epsilon}\left(  t\right)
$ the particle position, depending on $\epsilon$.

\begin{Corollary}
If
\[
\lim_{\epsilon\rightarrow0}\frac{\mu\left(  \epsilon\right)  }{\epsilon
}=\alpha\in\left[  0,\infty\right]
\]
then $x_{\epsilon}\left(  t\right)  $ converges in probability, uniformly in
$t\in\left[  0,T\right]  $, to the solution $x^{\alpha}\left(  t\right)  $ of
the stochastic equation in Stratonovich form (see Section 2.1)%
\[\left\{\begin{array}{l}
\ds{dx^{\alpha}\left(  t\right)    =\overline{u}\left(  x^{\alpha}\left(
t\right)  \right)  dt+\sum_{k\in K}\xi_{k}\left(  x^{\alpha}\left(  t\right)
\right)  \circ dw^{k}\left(  t\right)  -b_{\alpha}\left(  x^{\alpha}\left(
t\right)  \right)  dt}\\[10pt]
\ds{x^{\alpha}\left(  0\right)     =x_{0}}%
\end{array}\right.\]
where%
\begin{equation}
b_{\alpha}\left(  x\right)  =\frac{1}{2}\frac{\alpha}{c_{0}k_{T}\left(
x\right)  +\alpha}\left(\,  \sum_{k\in K}D\xi_{k}\left(  x\right)  \xi
_{k}\left(  x\right)  +C\left(  x\right)  c_{0}^{-1}\nabla\log k_{T}\left(
x\right)  \right)  \label{Ito drift}%
\end{equation}
and where the matrix $C\left(  x\right)  $ is given by 
\[C\left(  x\right)
=\sum_{k\in K}\xi_{k}\left(  x\right)  \otimes\xi_{k}\left(  x\right).\] 
\end{Corollary}

\begin{proof}
The comparison between the equation for $x^{\alpha}\left(  t\right)  $ in this
section and the one in Theorem \ref{teo2} is straightforward
for what concerns the terms $\sum_{k\in K}\xi_{k}\left(  x^{\alpha}\left(
t\right)  \right)  \circ dw^{k}\left(  t\right)  $ and $\overline{u}\left(
x^{\alpha}\left(  t\right)  \right)  dt$. In the expression for the drift
$b_{\alpha}\left(  x\right)  $ the factor $\frac{1}{2}\frac{\alpha}{c_{0}%
k_{T}\left(  x\right)  +\alpha}$ can be easily recognized. Let us only explain
why we have rewritten the term (of Section 2.1)
\[
\text{Tr}\left[  D\left(  \left(  \lambda^{-1}\sigma\right)  \left(  x\right)
B\right)  \left(  \lambda^{-1}\sigma\right)  \left(  x\right)  B\right]
+\left(  \sigma\left(  x\right)  B\right)  \left(  \sigma\left(  x\right)
B\right)  ^{\ast}\frac{\nabla\lambda\left(  x\right)  }{\lambda^{3}\left(
x\right)  }%
\]
in the form%
\[
\sum_{k\in K}D\xi_{k}\left(  x\right)  \xi_{k}\left(  x\right)  +C\left(
x\right)  c_{0}^{-1}\nabla\log k_{T}\left(  x\right)  .
\]
For this purpose, we have to make explicit the link between the notations of the
abstract theory and the example:\ the function $\lambda\left(  x\right)  $ is
$c_{0}k_{T}\left(  x\right)  $, the matrix $B$ is the identity, and the matrix
function $\sigma\left(  x\right)  $ has $\left\vert K\right\vert $ columns and
$d$ rows, and%
\[\left\{\begin{array}{l}
\ds{\sigma\left(  x\right)  z   =\lambda\left(  x\right)  \sum_{k\in K}\xi
_{k}\left(  x\right)  z^{k},\qquad z=\big(  z^{k}\big)  _{k\in K}}\\[10pt]
\ds{\left(  \sigma\left(  x\right)  ^{\ast}v\right)  _{k}   =\lambda\left(
x\right)  \xi_{k}\left(  x\right)  \cdot v,\qquad v\in\mathbb{R}^{d}.}
\end{array}\right.\]
Then, for the second term,
\[\frac{\nabla\lambda\left(  x\right)  }%
{\lambda\left(  x\right)  }=c_{0}^{-1}\nabla\log k_{T}\left(  x\right),\]
we have
\begin{align*}
\frac{\left(  \sigma\left(  x\right)  B\right)  \left(  \sigma\left(
x\right)  B\right)  ^{\ast}}{\lambda^{2}\left(  x\right)  }v  & =\sum_{k\in
K}\xi_{k}\left(  x\right)  \left(  \xi_{k}\left(  x\right)  \cdot v\right)
 =\left(  \sum_{k\in K}\xi_{k}\left(  x\right)  \otimes\xi_{k}\left(
x\right)  \right)  v=C\left(  x\right)  v
\end{align*}
and therefore, summarizing,
\[
\left(  \sigma\left(  x\right)  B\right)  \left(  \sigma\left(  x\right)
B\right)  ^{\ast}\frac{\nabla\lambda\left(  x\right)  }{\lambda^{3}\left(
x\right)  }=C\left(  x\right)  c_{0}^{-1}\nabla\log k_{T}\left(  x\right)  .
\]

Moreover, for the first term, with similar notations,
\[
\left(  \lambda^{-1}\sigma\right)  \left(  x\right)  Bz=\sum_{k\in K}\xi
_{k}\left(  x\right)  z^{k},\qquad D\left(  \left(  \lambda^{-1}%
\sigma\right)  \left(  x\right)  B\right)  z=\sum_{k\in K}D \xi_{k}\left(
x\right)  z^{k}%
\]
which gives, after a few computations,%
\[
\text{Tr}\left[  D\left(  \left(  \lambda^{-1}\sigma\right)  \left(  x\right)
B\right)  \left(  \lambda^{-1}\sigma\right)  \left(  x\right)  B\right]
=\sum_{k\in K}D\xi_{k}\left(  x\right)  \xi_{k}\left(  x\right)  .
\]

\end{proof}

We call the vector field $b_\alpha \left(  x\right)  $ the
\textit{inertial-It\^{o}-drift}. The motion of the particle contains three
components. One is the average motion $\overline{u}\left(  x^{\alpha}\left(
t\right)  \right)  dt$. The second one is the random motion \textit{as a
tracer} along the vector fields $\xi_{k}$:\ since the integration is in the
Stratonovich term (which heuristically corresponds to classical calculus), a
particle just subject to the term 
\[\sum_{k\in K}\xi_{k}\left(  x^{\alpha
}\left(  t\right)  \right)  \circ dw^{k}\left(  t\right)  \] would move under
the action of the fields $\xi_{k}$, just transported by them, in both
directions (namely by $\xi_{k}$ and $-\xi_{k}$). Finally, superimposed to
these two motions, which are clearly expected on the basis of elementary
arguments, there is the much less trivial drift $-b_{\alpha}\left(  x^{\alpha
}\left(  t\right)  \right)  dt$. It keeps memory of the inertia and its
interaction with the other elements of the motion.

\subsubsection{Centrifugal drift\label{Subsect just centrifugal}}

Consider the model with $\mu>0$. If the drivers $\xi_{k}\left(  x\right)  $
have rotational properties, the terms $\xi_{k}\left(  x\right)  z^{k}\left(
t\right)  $ produce fast rotations, alternated in the clockwise and
counterclockwise directions. A tracer, namely a particle without mass (without
inertia), would only rotate adhered to the fluid. But a particle with mass,
with inertia, should also perform a motion "outwards" with respect to the
center of rotation, due to centrifugal force.

In the zero-mass limit $\mu\rightarrow0$, the concept of inertia and
centrifugal force disappear, but the effect of such forces at the Newtonian
level remains in the "centrifugal drift". Our scaling limit theorem give us
the drift, $\sum_{k}D\xi_{k}\left(  x\right)  \xi_{k}\left(  x\right)  $,
which corresponds to this centrifugal effect. Let us show this by two
examples, one in this subsection and the other in the next one.

\begin{Proposition}
If $d=2$, $K$ is a singleton, $f$ is a bounded function with bounded
continuous second derivatives and
\[
\xi\left(  x\right)  =2f^{\prime}\left(  \left\vert x\right\vert ^{2}\right)
x^{\perp}%
\]
(where $x^{\perp}=\left(  -x_{2},x_{1}\right)  $ if $x=\left(  x_{1}%
,x_{2}\right)  $), then%
\[
-D\xi\left(  x\right)  \xi\left(  x\right)  =4f^{\prime}\left(  \left\vert
x\right\vert ^{2}\right)  ^{2}x.
\]

\end{Proposition}

Namely, in the drift $-b_{\alpha}\left(  x^{\alpha}\left(  t\right)  \right)
dt$, the term $-D\xi\left(  x^{\alpha}\left(  t\right)  \right)  \xi\left(
x^{\alpha}\left(  t\right)  \right)  dt$ contributes with a radial force
directed outward, corresponding to the intuition that the inertial particle,
due to rotation and inertia, has to move outward.

\begin{proof}
We have%
\[
D\xi\left(  x\right)  =\left(
\begin{array}
[c]{cc}%
-4f^{\prime\prime}\left(  \left\vert x\right\vert ^{2}\right)  x_{1}x_{2} &
-4f^{\prime\prime}\left(  \left\vert x\right\vert ^{2}\right)  x_{2}%
^{2}-2f^{\prime}\left(  \left\vert x\right\vert ^{2}\right)  \\[10pt]
4f^{\prime\prime}\left(  \left\vert x\right\vert ^{2}\right)  x_{1}%
^{2}+2f^{\prime}\left(  \left\vert x\right\vert ^{2}\right)   & 4f^{\prime
\prime}\left(  \left\vert x\right\vert ^{2}\right)  x_{1}x_{2}%
\end{array}
\right)
\]%
and
\begin{align*}
&  D\xi\left(  x\right)  \xi\left(  x\right)  =\left(
\begin{array}
[c]{c}%
4f^{\prime\prime}x_{1}x_{2}2f^{\prime}x_{2}-4f^{\prime\prime}x_{2}%
^{2}2f^{\prime}x_{1}-2f^{\prime}2f^{\prime}x_{1}\\[10pt]
-2f^{\prime}x_{2}4f^{\prime\prime}x_{1}^{2}-2f^{\prime}x_{2}2f^{\prime
}+4f^{\prime\prime}x_{1}x_{2}2f^{\prime}x_{1}%
\end{array}
\right)    =-4f^{\prime}\left(  \left\vert x\right\vert ^{2}\right)  ^{2}x
\end{align*}
where, to shorten the notations, in the matrix we have written only
$f^{\prime}$ and $f^{\prime\prime}$ in place of $f^{\prime}(  \left\vert
x\right\vert ^{2})  $ and $f^{\prime\prime}(  \left\vert
x\right\vert ^{2})  $ respectively. 
\end{proof}

\subsubsection{Centrifugal drift with concentration}

Associated with the centrifugal drift, there is often a concentration effect of
particles in sub-regions of space. The heuristic reason is the following one.
Assume that, instead of an isolated vortex as in the example of Section
\ref{Subsect just centrifugal}, the full space is occupied by vortices. Each
one has a tendency to "expel" inertial particles, exactly as in the case of a
single vortex. Thus, particles will be trapped in an intermediate region
between the vortices. This is the effect we want to describe in a very simple
geometrical case. 

Consider the case when $d=2$, $\overline{u}=0$, $K$ is a singleton,
$c_{0}k_{T}\left(  x\right)  $ is a constant $\lambda$, and we have a single
field $\xi_{k}\left(  x\right)  $, that we still index by $k$ with the meaning
of a wave number $k=\left(  k_{1},k_{2}\right)  $:
\[\left\{\begin{array}{l}
\ds{dx^{\alpha}\left(  t\right)     =\xi_{k}\left(  x^{\alpha}\left(  t\right)
\right)  \circ dw\left(  t\right)  -b_{\alpha}\left(  x^{\alpha}\left(
t\right)  \right)  dt}\\[10pt]
\ds{x^{\alpha}\left(  0\right)    =x_{0}}%
\end{array}\right.\]
where%
\[
b_{\alpha}\left(  x\right)  =\frac{1}{2}\frac{\alpha}{\lambda+\alpha}D\xi
_{k}\left(  x\right)  \xi_{k}\left(  x\right)  .
\]
Choose the fields studied by Greengard-Thomann:%
\[\psi_{k}\left(  x\right)   =\sin\left(  k_{1}x_{1}\right)  \cos\left(
k_{2}x_{2}\right)\]  
and
\[\xi_{k}\left(  x\right)     =\nabla^{\perp}\psi_{k}\left(  x\right)  =\left(
k_{2}\sin\left(  k_{1}x_{1}\right)  \sin\left(  k_{2}x_{2}\right)  ,k_{1}%
\cos\left(  k_{1}x_{1}\right)  \cos\left(  k_{2}x_{2}\right)  \right).\]


The vector field $\xi_{k}\left(  x\right)  $ is tangential to the level curves of $\psi_k$.
Hence, one can prove that the dynamics without drift%
\[
dy^{\alpha}\left(  t\right)  =\xi_{k}\left(  y^{\alpha}\left(  t\right)
\right)  \circ dw\left(  t\right)
\]
has trajectories super-imposed on these level curves. 

On the contrary, the drift $-b_{\alpha}\left(  x\right)  $ is not tangential
to such curves and moves $x^{\alpha}\left(  t\right)  $ "from the inside
towards the outside" of each vortex (but without crossing the vertical and
horizontal separatrix lines). Let us formalize this intuition. 

\begin{Proposition}
The process $\psi_{k}\left(  x^{\alpha}\left(  t\right)  \right)  $ is
pathwise monotone, decreasing when it lives in a region with $\psi_{k}>0$,
increasing where $\psi_{k}<0$. 
\end{Proposition}

The trajectories $x^{\alpha}\left(  t\right)  $ (the "particles") tend to
concentrate in the regions
\[
\Lambda_{\delta}=\left\{  x:\psi_{k}\left(  x\right)  \in\left[
-\delta,\delta\right]  \right\}
\]
that are (relatively) thin regions around the vertical and horizontal
separatrix lines. 

\begin{proof}
By It\^{o} formula in Stratonovich form,%
\begin{align*}
d\psi_{k}\left(  x^{\alpha}\left(  t\right)  \right)    & =\left(  \nabla
\psi_{k}\right)  \left(  x^{\alpha}\left(  t\right)  \right)  \cdot\xi
_{k}\left(  x^{\alpha}\left(  t\right)  \right)  \circ dw\left(  t\right)
-\left(  \nabla\psi_{k}\right)  \left(  x^{\alpha}\left(  t\right)  \right)
\cdot b_{\alpha}\left(  x^{\alpha}\left(  t\right)  \right)  dt\\
& =-\left(  \nabla\psi_{k}\right)  \left(  x^{\alpha}\left(  t\right)
\right)  \cdot b_{\alpha}\left(  x^{\alpha}\left(  t\right)  \right)  dt
\end{align*}
because $\xi_{k}=\nabla^{\perp}\psi_{k}$. Moreover,%
\[
D\xi_{k}\left(  x\right)  =\left(
\begin{array}
[c]{cc}%
k_{1}k_{2}\cos\left(  k_{1}x_{1}\right)  \sin\left(  k_{2}x_{2}\right)   &
k_{2}^{2}\sin\left(  k_{1}x_{1}\right)  \cos\left(  k_{2}x_{2}\right)  \\[10pt]
-k_{1}^{2}\sin\left(  k_{1}x_{1}\right)  \cos\left(  k_{2}x_{2}\right)   &
-k_{1}k_{2}\cos\left(  k_{1}x_{1}\right)  \sin\left(  k_{2}x_{2}\right),
\end{array}
\right)
\]%
and \[
D\xi_{k}\left(  x\right)  \xi_{k}\left(  x\right)  =k_{1}k_{2}\left(
k_{2}\sin\left(  k_{1}x_{1}\right)  \cos\left(  k_{1}x_{1}\right)  ,-k_{1}%
\sin\left(  k_{2}x_{2}\right)  \cos\left(  k_{2}x_{2}\right)  \right),
\]%
and 
\[
\nabla\psi_{k}\left(  x\right)  =\left(  k_{1}\cos\left(  k_{1}x_{1}\right)
\cos\left(  k_{2}x_{2}\right)  ,-k_{2}\sin\left(  k_{1}x_{1}\right)
\sin\left(  k_{2}x_{2}\right)  \right)  .
\]
This implies
\begin{align*}
\nabla\psi_{k}\left(  x\right)  \cdot D\xi_{k}\left(  x\right)  \xi_{k}\left(
x\right)    & =\left(  k_{1}k_{2}\right)  ^{2}\sin\left(  k_{1}x_{1}\right)
\cos\left(  k_{2}x_{2}\right)   =\left(  k_{1}k_{2}\right)  ^{2}\psi_{k}\left(  x\right),
\end{align*}
and hence%
\[
d\psi_{k}\left(  x^{\alpha}\left(  t\right)  \right)  =-\left(  k_{1}%
k_{2}\right)  ^{2}\psi_{k}\left(  x\right)  dt.
\]
This implies the claims.
\end{proof}

Another indicator related to concentration, but of a different type, is given by
the sign of the divergence of $-b_{\alpha}\left(  x\right)  $ (the divergence
of the Stratonovich noise is zero). Since
\[
D\xi_{k}\left(  x\right)  \xi_{k}\left(  x\right)  =k_{1}k_{2}\left(
k_{2}\sin\left(  k_{1}x_{1}\right)  \cos\left(  k_{1}x_{1}\right)  ,-k_{1}%
\sin\left(  k_{2}x_{2}\right)  \cos\left(  k_{2}x_{2}\right)  \right)
\]
we have%
\begin{align*}
-\operatorname{div}b_{\alpha}\left(  x\right)    & =-\left(  k_{1}%
k_{2}\right)  ^{2}\left[  \cos^{2}\left(  k_{1}x_{1}\right)  -\sin^{2}\left(
k_{1}x_{1}\right)  -\cos^{2}\left(  k_{2}x_{2}\right)  +\sin^{2}\left(
k_{2}x_{2}\right)  \right]  \\[10pt]    
& =-2\left(  k_{1}k_{2}\right)  ^{2}\left[  \sin^{2}\left(  k_{2}x_{2}\right)
-\sin^{2}\left(  k_{1}x_{1}\right)  \right]  .
\end{align*}
Where the sign of $-\operatorname{div}b_{\alpha}\left(  x\right)  $ is
negative, the trajectories $x^{\alpha}\left(  t\right)  $ that are close each
others become closer, namely, there is a form of concentration, not explicit in
the original second-order model driven only by divergence-free fields. The
level curves of $-\operatorname{div}b_{\alpha}\left(  x\right)  $ are given in
the next picture, in the example $k=\left(  1,1\right)  $. %


\begin{Remark}
{\em There are (important) cases where the centrifugal effect is cancelled by other
elements. Take $d=2$, $\lambda$ constant,%
\[
\xi_{k}\left(  x\right)  =\frac{k^{\perp}}{\left\vert k\right\vert }e^{ik\cdot
x}%
\]
with little abuse of complex-valued notations, just to make the exposition
shorter. These are translational fields, each of them directed only in one
direction, without rotation. The cumulative effect could be of a rotation but
in a totally random position and direction. Using the complex notations,%
\[
C\left(  x,y\right)     =\sum_{k}\xi_{k}\left(  x\right)  \otimes
\overline{\xi_{k}\left(  y\right)  }=\sum_{k}\frac{k^{\perp}\otimes k^{\perp}%
}{\left\vert k\right\vert ^{2}}e^{ik\cdot\left(  x-y\right)  },\]
so that
\[C\left(  x\right)    =\sum_{k}\frac{k^{\perp}\otimes k^{\perp}}{\left\vert
k\right\vert ^{2}},\ \ \ \ \ 
\left(  \nabla\cdot C\right)  _{i}\left(  x\right)     =0,\qquad
\operatorname{div}b\left(  x\right)  =0.\]
Therefore, no centrifugal effect emerges.}
\end{Remark}

\subsubsection{Turbo-phoretic drift}

The second main drift is behind the turbo phoretic effect. It is a drift specific to
turbulence and again due to inertia:\ particles in turbulent fluids, in the
presence of a non-zero gradient of turbulence intensity $\nabla k_{T}\left(
x\right)  $, have a tendency to move in the regions of lower turbulence
intensity. The phenomenon can be understood by analogy with Brownian motion: very small-scale turbulence acts on particles much as molecules act on a Brownian particle, kicking them randomly in all directions. When this random agitation is more intense on one sub-region than on its complement, the imbalance produces a net force: the particle is kicked out of the high-intensity region more vigorously than it is kicked back in from the surroundings, and on average it drifts toward regions of lower turbulent intensity. The mechanism is analogous to that of a pressure gradient. It is
similar to the concept of pressure gradient. In order to apply, it is
necessary that the turbulence eddies are very small compared to the particle.
Therefore, the turbo phoretic effect applies to particles that are relatively
larger with respect to the turbulent eddies. This corresponds to our
multiscale assumption that $u^{\prime\prime}\left(  x,t\right)  $ (the term
responsible for $k_{T}\left(  x\right)  $) corresponds to scales much smaller
than the particle scale. 

\begin{claim}
{\em The term, in the drift $-b_{\alpha}\left(  x\right)  $, corresponding to the
turbo-phoretic effect is the term$\,$%
\[
-C\left(  x\right)  c_{0}^{-1}\nabla\log k_{T}\left(  x\right).
\]
}
\end{claim}

\begin{Remark}
{\em If $C\left(  x\right)  $ is just the identity matrix, the vector $-\nabla\log
k_{T}\left(  x\right)  $ (recall our convention that the drift is $-b_{\alpha
}$) points in the direction of lower values of $k_{T}\left(  x\right)  $, as
just described. If $C\left(  x\right)  $, at some $x$, is not the identity, we
may understand the role of multiplication by the matrix $C\left(
x\right)  $ by Principal Component Analysis:\ if the vector $\nabla\log
k_{T}\left(  x\right)  $ is aligned with the eigenvector of $C\left(
x\right)  $ having the highest eigenvalue (the so-called principal component),
then the effect of $\nabla\log k_{T}\left(  x\right)  $ is amplified,
relative to what happens if it is aligned with other eigenvectors or just
generic. The principal component is the direction in which the random
fluctuations are more intense, hence pushing the particle faster towards the
low-level regions of $k_{T}\left(  x\right)  $.}
\end{Remark}

As an example of turbo phoresis, consider a turbulent pipe flow. Strictly
speaking, this example deviates from our assumption that the model lives in
full space, but extending the results to such a simple geometry would require
only minor modifications. 

Consider an infinite 2D pipe, namely a set in $\mathbb{R}^{2}$ of the form%
\[
D=\mathbb{R}\times\left[  -1,1\right]  .
\]
We write $x=\left(  x_{1},x_{2}\right)  $. We assume the fluid is turbulent in
the pipe and composed of an average part $\overline{u}\left(  x\right)  $ of
the form%
\[
\overline{u}\left(  x\right)  =\left(  \overline{v}\left(  x_{2}\right)
,0\right)
\]
with
\[
\overline{v}\left(  \pm1\right)     =0,\ \ \ \ \ 
\overline{v}\left(  x_{2}\right)    >0,\ \text{ for all }x_{2}\in\left(
-1,1\right).
\]




The turbulent kinetic energy $k_{T}\left(  x\right)  $ does not have an
explicit form, but its shape looks qualitatively similar, if the radius of the
pipe is not so large (otherwise there is a depletion in the central zone.

Assume that the random part $u^{\prime}\left(  x,t\right)  $ of the turbulent
fluid, the one at a scale comparable with the particle scale, is essentially
space-homogeneous; this is just to simplify the exposition. In such a case, we
may assume $C\left(  x\right)  =Id$.

Under these conditions, the inertial-It\^{o}-drift $\,-b_{\alpha}\left(
x\right)  $ does not contain the component \[-\sum_{k}D\xi_{k}\left(  x\right)
\xi_{k}\left(  x\right),\]
 but contains a non-trivial component%
\[
-\frac{1}{2}\frac{\alpha}{c_{0}k_{T}\left(  x\right)  +\alpha}\nabla\log
k_{T}\left(  x\right)  .
\]
Taking in account the sign, we see that particles drift in the direction of
low values of $k_{T}\left(  x\right)  $. It means they drift towards the
boundary of the pipe. This is a well-known experimental fact \cite{Reeks},
\cite{Johnson}. It is interesting to remark that such an effect is observed for
heavy particles, while the lightest ones may remain in the central zone; this
corresponds to our assumption that the particle scale is much larger than the
scale of the smallest turbulent eddies $u^{\prime\prime}\left(  x,t\right)  $,
which therefore could be modeled by means of the Boussinesq approximation.
Turbo phoresis may also contribute to particle clustering and thus to increase
the rate of coalescence \cite{Delillo}.

\begin{Remark}
{\em Both the term $\sum_{k}D\xi_{k}\left(  x\right)  \xi_{k}\left(  x\right)  $
corresponding to the centrifugal force and the term $C\left(  x\right)
\nabla\log k_{T}\left(  x\right)  $ corresponding to the turbo phoretic effect
are multiplied by the factor $\frac{1}{2}\frac{\alpha}{c_{0}k_{T}\left(
x\right)  +\alpha}$. It is zero in the case when inertia (namely $\mu$) is
infinitesimal with respect to the time scale $\epsilon$ of $u^{\prime}\left(
x,t\right)  $. Otherwise, it increases with $\alpha$, the limit of $\frac
{\mu}{\epsilon}$. A large $k_{T}\left(  x\right)  $ depletes the effect, since
viscosity reduces the effects of inertia.}
\end{Remark}

\section*{Acknowledgements} 
This material is based upon work supported by the National Science
Foundation under Grant No. DMS-2424139, while the  authors were in
residence at the Simons Laufer Mathematical Sciences Institute in
Berkeley, California, during the Fall 2025 semester.


\begin{thebibliography}{99}

\bibitem {Bec}J. Bec, K. Gustavsson, B. Mehlig, {\em Statistical models for the
dynamics of heavy particles in turbulence}, Annu. Rev. Fluid Mech. 56 (2024), pp. 189--213.

\bibitem {Berselli}C. L. Berselli, T. Iliescu, J. W. Layton, {\sc Mathematics of
Large Eddy Simulation of Turbulent Flows}, Springer, 2006.

\bibitem {Bousinnesq}J. Boussinesq, {\em Essai sur la th\'{e}orie des eaux
courantes}, M\'{e}moires pr\'{e}sent\'{e}s par divers savants a l'Academie des
Sciences de l'Institut National de France, XXIII (1), (1877).


\bibitem{cerrai-09} S.~Cerrai, {\em A Khasminskii type averaging principle for stochastic reaction-diffusion equations},  Annals of Applied Probability 19 (2009), pp. 899--948.

\bibitem {Delillo}F.~De Lillo, M.~Cencini, S.~Musacchio, G.~Boffetta,
{\em Clustering and turbophoresis in a shear flow without walls}, Physics of Fluids
28 (2016).



\bibitem{fgp} F.~Flandoli, M.~Gubinelli, E.~Priola, {\em Well-posedness of the transport equation by stochastic
              perturbation}, Inventiones 180 (2010), pp. 1--53.


\bibitem {FlaHuang}F.~Flandoli, R.~Huang, {\em Coagulation dynamics under
environmental noise: Scaling limit to SPDE}, ALEA, Latin American Journal of Probability and Mathematical
              Statistics 19 (2022), pp. 1241--1292.

\bibitem {FlaPalmViv}F. Flandoli, M. Palmieri, M. Viviani, {\em The early stage of
the motion along the gradient of a concentrated vortex structure}, arXiv:2505.22700.

\bibitem{freidlin} M.~Freidlin, {\em  Some remarks on the Smoluchowski-Kramers approximation}, Journal of Statistical Physics 117 (2004), pp. 617-634.






\bibitem{hmdvw}S.~Hottovy, A.~McDaniel, G.~Volpe, J.~Wehr,
{\em The Smoluchowski-Kramers limit of stochastic differential equations with arbitrary state-dependent friction}, Communications in Mathematical Physics 336 (2015), pp. 1259--1283.

\bibitem{Wehr2} S.~Hottovy, G.~Volpe, J.~Wehr, {\em Noise-induced drift in stochastic differential equations
with arbitrary friction and diffusion
in the Smoluchowski-Kramers limit}, Journal of  Statistical Physics 146  (2012), pp. 762--773.

\bibitem {Johnson}P.~L.~Johnson, M.~Bassenne, P.~Moin, {\em Turbophoresis of small
inertial particles: theoretical considerations and application to
wall-modelled large-eddy simulations}, Journal of Fluid Mechanics 883 (2020).



\bibitem{pav-stu} R.~Kupferman, G.~A.~Pavliotis, A.~M.~Stuart, {\em It\^o versus Stratonovich white-noise limits for systems with inertia and colored multiplicative noise}, Physical Review E 70 (2004), 036120, 9 pp.



\bibitem {Kraichnan}R. H. Kraichnan, {\em Small-scale structure of a scalar field
convected by turbulence}, The Physics of Fluids 11 (1968), pp. 945--953.


\bibitem {Majda}A. J. Majda, I. Timofeyev, E. Vanden Eijnden, {\em A mathematical
framework for stochastic climate models}, Comm. Pure Appl. Math.,
54 (2001), pp. 891--974.


\bibitem{PS-2} G.~Pavliotis, A.~Stuart, {\em Analysis of white noise limits for stochastic systems with two fast relaxation times}, SIAM Multiscale Modelling and Simulation 4 (2005), pp. 1--35.


\bibitem {Reeks}M.~W.~Reeks, {\em The transport of discrete particles in
inhomogeneous turbulence}, Journal of Aerosol Science 14 (1983), pp. 729--739.










\end{thebibliography}
\end{document}